\documentclass[leqno]{article}
\usepackage{amsmath}
\usepackage{amssymb}
\usepackage{amsthm}                                         
\usepackage{url}
\usepackage{graphicx}
\usepackage{epsfig}
\usepackage[dvipsnames]{xcolor}
\newtheorem{theorem}{\bf Theorem}

\newtheorem{corollary}[theorem]{\bf Corollary}

\newtheorem{proposition}[theorem]{\bf Proposition}
\newtheorem{definition}[theorem]{\bf Definition}

\newtheorem{remark}[theorem]{\bf Remark}

\numberwithin{equation}{section}
\numberwithin{theorem}{section}
\numberwithin{figure}{section}

\def\R{\mathbb{R}}
\def\L{\mathbb{L}}
\def\R{\mathbb{R}}

\begin{document}
\renewcommand{\thefootnote}{}
\footnotetext{The authors were partially supported by MICINN-FEDER, Grant No. MTM2016- 80313-P, Junta de Andaluc\'ia Grant No. FQM325 }

\title{A Calabi's Type Correspondence} 
\author{Antonio Mart\'{\i}nez and A. L. Mart\'inez-Trivi\~no}
\vspace{1in}
\date{}
\maketitle
\noindent {\footnotesize Department of Geometry and Topology, University of Granada, E-18071 Granada, Spain.\\ 
e-mails: amartine@ugr.es, aluismartinez@ugr.es}
\begin{abstract}
\noindent Calabi  observed that  there is a natural  correspondence between the 
solutions of the minimal surface equation in $\R^3$  with those of the maximal spacelike  surface equation in $\L^3$. We are going to show how this correspondence can be extended to the family  of $\varphi $-minimal  graphs in $\R^3 $ when the function  $\varphi$ is invariant under a two-parametric group of translations. We give also applications in the study and description of new examples.
\end{abstract}


\everymath={\displaystyle}
\section{Introduction}
Differential geometry of surfaces and partial differential equations (PDEs) have a strong link by means of which both theories benefit mutually. Actually, many classic PDEs are linked to interesting geometric problems. The geometry allows to integrate these equations, to establish non trivial properties of the solutions and to give some superposition principles which determine new solutions in terms of already known solutions.
The classical theory of surfaces shows that geometric transformations may also be used to construct new surfaces from a given one.

\

Two of the most studied geometric PDEs has been
\begin{equation}\label{mle}
div\left(\frac{\nabla u}{\sqrt{1 + \epsilon |\nabla u|^2}}\right)= 0, \qquad \epsilon\in\{1,-1\}.
\end{equation}
They are the Euclidean minimal equation ($\epsilon= 1$) and the  Lorentzian spacelike maximal equation ($\epsilon = -1$). Calabi proved in \cite{Calabi}, there is a natural  one to one correspondence between the  solutions of \eqref{mle} with $\epsilon= 1$ and the solutions of \eqref{mle} with $\epsilon=-1$. This relationship gives a (local) connection  between Euclidean minimal graphs and Lorentzian spacelike maximal graphs which is useful  for  describing  examples and for applying similar methods to the study of their geometrical and topological properties.  What we offer in this work is an extension of this correspondence to solutions of

\begin{equation}
div\left(\frac{\nabla u}{\sqrt{1 + \epsilon |\nabla u|^2}}\right)= \frac{G(u)}{\sqrt{1 + \epsilon |\nabla u|^2}}, \qquad \text{ in $\Omega_\epsilon\subseteq \R^2$},\label{leeq}
\end{equation}
where $\Omega_\epsilon$ is a simply connected  domain, $G(u)$ is a function of $u$ and  $ \epsilon\in\{1,-1\}$.

\

For $\epsilon=1$, \eqref{leeq} describes the equilibrium condition of a graph in $\R^3$ in a external conservative force field whose potential is invariant under a two-parameter group of translations. When $G(u)\equiv 1$, the graphs of a solutions of \eqref{leeq} are {\sl translating solitons}, that is, solutions of the mean curvature flow which move by either a vertical translation in $\R^3$ if $\epsilon=1$ or in $\L^3$ if  $\epsilon=-1$. 

The paper is organized as follows,   Sections 2 and 3 deal  with  some elementary  facts of the theory of $\varphi$-minimal surfaces in $\R^3$ and spacelike $\varphi$-maximal  surfaces in $\L^3$.
In Section 4, we give a  Calabi's type corresponde between   the solutions  of \eqref{leeq} with $\epsilon= 1$ and the solutions of \eqref{leeq} with $\epsilon=-1$. Finally, Section 5 is devoted to give some applictions in the description and study of new examples. 

\section{ The $\varphi$-minimal equation}
The equilibrium condition for a surface $\Sigma$ in $\R^3$ in a external force field ${\cal F}$ was described by Poisson \cite[pp. 173-187]{Poisson}.  When  the intrinsic forces of the surface  are assumed to be equal,  ${\cal F}$ must be conservative and it writes as $ {\cal F}= \overline{\nabla} (\mathrm{e}^\varphi),$ for some smooth function $\varphi$ on a domain of $\R^3$ which contains  $\Sigma$.   In this case,   the equilibrium condition is given  in terms of the mean curvature vector  ${\text {\bf H}}$ of $\Sigma$ as follows: 
\begin{align} 
&{\text {\bf H}} = \left(\overline{\nabla} \varphi \right)^\perp \label{fminimal}
\end{align}
where $ \overline{\nabla} $ is  the gradient operator in $\R^3$ and  $\perp$ denotes the projection to the normal bundle of  $\Sigma$.

\

A surface  $\Sigma$  satisfying \eqref{fminimal} is called  $\varphi$-{\sl minimal} and it can be also viewed  either  as  a critical point of the weighted volume functional 
\begin{equation}
V_{\varphi} (\Sigma):= \int_\Sigma  \mathrm{e}^\varphi\ dA_\Sigma,
\end{equation}
where $dA_\Sigma$ is the volume element of $\Sigma$, or as a minimal surface in  the conformally changed metric 
\begin{equation} 
G_\varphi:=  \mathrm{e}^\varphi\ \langle . , . \rangle.
\end{equation}

%In this paper, we will assume that $\varphi$  only depend on the third coordinate, that is, it is invariant under a two-parameter group of horizontal translations. Under this assumption will show how to extend the Calabi's correspondence to  $\varphi$-minimal graphs

For the particular case that the function  $\varphi$ is invariant under a two-parameter group of translations in $\R^3$ we have  some interesting examples: 
\begin{itemize}
\item {\sl Minimal surfaces},  if   $\varphi$ is  constant.
\item  {\sl Translating solitons}, if  $\varphi $ is a linear function. 
\item {\sl Singular minimal surfaces} (also called {\sl cupolas} because they arise by considering  ${\cal F}$ as  the gravitacional force field) if $\varphi (p) = \alpha \log( \langle p , \vec{v}\rangle$), where  $\alpha\in \R$ and $\vec{v}$ is a constant vector. 
\end{itemize}
Some references in the study of the above examples are  \cite{R1,MSHS2,MSHS1,Os}.

\

When  $\varphi$ is invariant under a two-parameter group of translations then, up to a motion in $\R^3$, we can consider  that the external force field  ${\cal F}$ is always  a vertical vector field, that is, 
\begin{equation*}{\cal F} \wedge \vec{e_3} = \mathrm{e}^{\varphi} \overline{\nabla}\varphi \wedge \vec{e_3}\equiv  0, 
\end{equation*}
where $\vec{e_3}=(0,0,1)$ and   $\varphi$ {\sl only depends on the third coordinate} in $\R^3$. In this case, the condition \eqref{fminimal} writes as  
\begin{equation}
{\text {\bf H}}=  \dot{\varphi} \ \vec{e}_3^\perp,\label{vforce}
\end{equation}
where $(\ \dot{ }\ )$ denotes derivate respect to the third coordinate.
\begin{definition}
A $\varphi$-minimal surface $\Sigma$ whose mean curvature vector satisfies  \eqref{vforce} will be called  {\sl $[\varphi, \vec{e_3}]$-minimal}.
\end{definition}
Let $u:\Omega_1 \longrightarrow \R$ be a regular function on a simply connected planar domain $\Omega_1$ and consider  $ \psi(x,y) = (x,y,u)$ its graph. Then, the induced metric, Gauss map and mean curvature vector of $\psi$ write, respectively,  by 
\begin{align} 
& g:=  (1 + u_x^2)dx^2 + (1+u_y^2)dy^2 + 2 u_x u_y dx dy,\label{metric}
\\
& N:=  \frac{1}{W} ( -u_x,-u_y,1), \quad W=\sqrt{1 + u_x^2 + u_y^2},\label{normal}\\
& \text{\bf H}:= -\frac{1}{W^3} {\cal L}u \ N, \label{mc}
\end{align}
where,
\begin{align*}
& {\cal L}u = (1+u_x^2) u_{yy} + (1+u_y^2) u_{xx} - 2 u_x u_y u_{xy} ,  
\end{align*}
Hence and  from \eqref{vforce}, we get that $\psi$ is a $[\varphi, \vec{e}_3]$-minimal graph  if and only if 
 $u$ is a solution of the  following elliptic  partial differential equation:
\begin{equation}
 (1+u_x^2) u_{yy} + (1+u_y^2) u_{xx} - 2 u_x u_y u_{xy} =  \dot{\varphi}(u) W^2.  \label{fe}
\end{equation}
Using \eqref{fe}, we  can see
\begin{align*}
&\left( \frac{ 1 + u_y^2}{W} \mathrm{e}^{\varphi(u)}\right)_x -\left( \frac{u_x u_y}{W} \mathrm{e}^{\varphi(u)} \right)_y = - \frac{ u_x  \ \mathrm{e}^{\varphi(u)}{\cal L}u }{W^3} +\frac{u_x \ \mathrm{e}^{\varphi(u)} \dot{\varphi}(u) }{W} = 0,\\
&\left( \frac{ 1 + u_x^2}{W} \mathrm{e}^{\varphi(u)} \right)_y -\left( \frac{u_x u_y}{W} \mathrm{e}^{\varphi(u)} \right)_x= - \frac{ u_y \ \mathrm{e}^{\varphi(u)} {\cal L}u }{W^3} +\frac{u_y \ \mathrm{e}^{\varphi(u)} \dot{\varphi}(u) }{W} =0.
\end{align*}
The equation  \eqref{fe}, is equivalent to the integrability of the following differential system,
\begin{equation} 
\label{integrabilidad}
 \phi_{xx}=  \frac{ 1 +  u_x^2}{ W} \mathrm{e}^{\varphi (u)}, 
\quad 
\phi_{xy}=  \frac{ u_xu_y }{W} \mathrm{e}^{\varphi(u)}, 
 \quad   
\phi_{yy}=  \frac{ 1 +   u_y^2}{W} \mathrm{e}^{\varphi(u)},
\end{equation}
for a convex function  $\phi:\Omega_1 \longrightarrow \R$ (unique, module linear polynomials).
\\
From \eqref{metric} and \eqref{fe}, the laplacian operator $\Delta_g$ of $g$ is given by
\begin{align}
W^2 \Delta_g & :=  (1 + u_x^2)\partial^2_{yy} + (1+ u_y^2)\partial^2_{xx} - 2 u_x u_y \partial^2_{xy} - \dot{\varphi} (u)(u_x \partial_x + u_y \partial_y ),\label{laplacian}
\end{align}
and  a straightforward computation gives,
\begin{align}
& \Delta_g \psi = (\overline{\nabla}\varphi)^\perp = \langle N,\vec{e}_3\rangle \dot{\varphi}(u) \ N,
\label{l1}\\
& \Delta_g \phi_x =  \frac{\mathrm{e}^{\varphi(u)} \dot{\varphi}(u)}{W} u_x = \langle N,\vec{e}_3\rangle \mathrm{e}^{\varphi(u)} \dot{\varphi}(u) \ u_x, \label{l2}
\\
& \Delta_g \phi_y =  \frac{\mathrm{e}^{\varphi(u)} \dot{\varphi}(u)}{W} u_y = \langle N,\vec{e}_3\rangle \mathrm{e}^{\varphi(u)} \dot{\varphi}(u) \  u_y,\label{l3}
\\
& \Delta_g (\int \mathrm{e}^{\varphi(u)} du) = \mathrm{e}^{\varphi(u)} \dot{\varphi}(u)= \langle N,\vec{e}_3\rangle\mathrm{e}^{\varphi(u)} \dot{\varphi}(u) \  W. \label{l4}
\end{align}
\section{The $\varphi$-maximal equation}
Let $\L^3$ be the Minkowski space $\R^3$ with the Lorentz metric 
\begin{equation} \ll . ,  . \gg = dx^2 + dy^2 - dz^2.\label{mmetric}\end{equation}
A surface in $\mathbb{L}^{3}$ is called \textit{spacelike} if the induced metric on the surface is a positive definite Riemannian metric. This kind of surfaces have played a major role in Lorentzian geometry,  for a survey of some results we refer to \cite{Bar}.

\

A spacelike surface  $ {\widetilde \Sigma}$ in $\L^3$ is called  $\varphi$-{\sl maximal} if its mean curvature vector  $ {\widetilde{\text {\bf  H}}}$  satisfies
\begin{align} 
& {\widetilde {\text {\bf H}}} = \left( \overline{\nabla}^{\L^3} \varphi \right)^\perp,\label{fmaximal}
\end{align}
where $\overline{\nabla}^{\L^3}$ denotes   the the gradient operator in $\L^3$ and $\varphi$ is a smooth function on a domain in $\L^3$ containing $ {\widetilde \Sigma}$.

\

As in the Euclidean case, 
a  $\varphi$-maximal spacelike  surface   can be also viewed either   a critical point of the weighted volume functional 
\begin{equation}
{\widetilde V}_{\varphi} ({\widetilde \Sigma}):= \int_{{\widetilde \Sigma}}  \mathrm{e}^\varphi\ dA_{{\widetilde \Sigma}},
\end{equation}
 or a  maximal (zero mean curvature) spacelike surface in  the conformally changed metric 
\begin{equation} 
{\widetilde G}_\varphi:=  \mathrm{e}^\varphi \ll.,.\gg.
\end{equation}
\begin{definition}
{\rm If $\varphi$ only depend on the third coordinate, any spacelike surface  with mean curvature vector satisfying   \eqref{fmaximal} will be called $[\varphi, \vec{e}_3]$-maximal.

\

Well known  examples of $[\varphi, \vec{e}_3]$-maximal  are  the maximal  surfaces and the  translating solitons, whose study is an exciting and already classical mathematical research field,  see \cite{CY,QD} for some results. In analogy to the Euclidean case, a  $[\varphi, \vec{e}_3]$-maximal  spacelike with $\varphi(p)= \alpha \log{<p,\vec{e}_3> }$, $\alpha\in \R$, $\alpha\neq 0$ will be called {\sl singular $\alpha$-maximal} surface.}
\end{definition}

By a strightforward computation, if $\Omega_{-1}$ is a simply connected planar domain, it is not difficult to prove that the vertical graph  in $\L^3$ of a function $\overline{u}:\Omega_{-1} \longrightarrow \R$ is a  $[{\varphi}, \vec{e}_3]$-maximal spacelike  if and only if $\overline{u}$ is a solution of the  following elliptic partial differential equation:
\begin{equation}
 (1-\overline{u}_x^2) \overline{u}_{yy} +  (1-\overline{u}_y^2) \overline{u}_{xx} + 2  \overline{u}_x  \overline{u}_x  \overline{u}_{xy} +\dot{\varphi} (\overline{u})\overline{W}^2=0, \label{Lfe}
\end{equation}
where  $\overline{W}= \sqrt{ 1-\overline{u}_x^2- \overline{u}_y^2}$.
From \eqref{mmetric} and \eqref{Lfe}, the laplacian operator $\Delta_{\overline{g}}$ of the induced metric $\overline{g}$  is given by
\begin{align}
\overline{W}^2 \Delta_{\overline{g}} & :=  (1 - \overline{u}_x^2)\partial^2_{yy} + (1-\overline{u}_y^2)\partial^2_{xx} + 2\overline{u}_x\overline{u}_y \partial^2_{xy} - \dot{\varphi}(\overline{u}) (\overline{u}_x \partial_x + \overline{u}_y \partial_y ),\label{llaplacian}
\end{align}
and   the equation  \eqref{Lfe}, is equivalent to the integrability of the following differential system,
\begin{equation}  \overline{\phi}_{xx}=  \frac{ 1 -   \overline{u}_x^2}{ \overline{W}} \mathrm{e}^{\varphi(\overline{u})}, 
\quad 
 \overline{\phi}_{xy}=  - \frac{ \overline{u}_x \overline{u}_y}{ \overline{W}} \mathrm{e}^{\varphi(\overline{u})}, 
 \quad   
\overline{\phi}_{yy}=  \frac{ 1 -   \overline{u}_y^2}{ \overline{W}} \mathrm{e}^{\varphi(\overline{u})}\label{lintegra}
\end{equation}
for a convex function  $\overline{\phi}:\Omega_{-1} \longrightarrow \R$ (unique, module linear polynomials).
\section{ The correspondence}
Calabi observed,  \cite{Calabi},   there is a natural  correspondence between all solutions of the minimal surface equation in $\R^3$  with those of the maximal spacelike  surface equation in $\L^3$. 
\vspace{0.2 cm}
We are going to show how this correspondence can be extended to the family of $[\varphi ,\vec{e_3}]$-minimal  graphs in $\R^3 $.
\begin{theorem} \label{th1} Let $\Omega_1$ be a simply connected planar domain,  $\psi: \Omega_1 \longrightarrow \R^3$, $\psi(x,y)= (x,y,u)$ be a $[\varphi ,\vec{e_3}]$-minimal  graph in $\R^3 $, $\phi$ be a solution to the system \eqref{integrabilidad} and $\vartheta$ be a primitive function of $\mathrm{e}^\varphi $ (that is, $\dot{\vartheta}=\mathrm{e}^\varphi)$. Then  $\widetilde{\psi}: \Omega_1 \longrightarrow \L^3$ given by
\begin{equation}
 \widetilde{\psi} := (\phi_x,\phi_y, \vartheta(u) ), \label{correspondencia1}
\end{equation}
is a $[-\varphi\circ \vartheta^{-1} ,\vec{e_3}]$-maximal  spacelike graph in the Lorentz-Minkowski space whose  Gauss map
$ \widetilde{N} $ writes
\begin{align}
\widetilde{N} &= (u_x, u_y, W), 
\label{phiL3}
\end{align}
\noindent The induced metrics $g$ and $\widetilde{g}$ of $\psi$ and $\widetilde{\psi}$, respectively,  are conformal and  the mean curvature $H$ ($\widetilde{H}$) and  Gauss curvature $K$ ($\widetilde{K}$ ) of  $\psi$ ($\widetilde{\psi}$) satisfy
\begin{align}
&  \widetilde{H} + W^2 \mathrm{e}^{-\varphi(u)}   H = 0,\label{hh}\\
&\widetilde{K}+ W^4  \mathrm{e}^{-2\varphi(u)}K = 0.\label{kk}
\end{align}
\end{theorem}
\begin{proof}
From \eqref{integrabilidad}, \eqref{phiL3} and \eqref{correspondencia1}, we have
\begin{align}
& \ll\widetilde{\psi}_x, \widetilde{N}\gg= \ll\widetilde{\psi}_y, \widetilde{N}\gg=0, \quad  \ll\widetilde{N}, \widetilde{N}\gg=-1,\\
& \widetilde{g}= \ll d\widetilde{\psi},d\widetilde{\psi}\gg= (d\phi_x)^2 + (d\phi_y)^2 - (d\vartheta)^2 = \frac{\mathrm{e}^{ 2\varphi(u)}}{W^2} g. \label{cmetric}
\end{align}
So,  $\widetilde{\psi}$ is spacelike graph with Gauss map $\widetilde{N}$. As $\widetilde{g}$ and $g$ are conformal metrics, from \eqref{l2}, \eqref{l3},  \eqref{phiL3} and \eqref{cmetric}, we get that 
$$ \widetilde{{\text {\bf H}}} = \Delta_{\widetilde{g}} \widetilde{\psi} = -\frac{W \dot{\varphi}(u)}{ \mathrm{e}^{\varphi(u)}}\widetilde{N} =-\frac{d\varphi}{d\vartheta}W\widetilde{N}= \left(- \overline{\nabla}^{\L^3}\varphi\circ \vartheta^{-1}\right)^\perp,$$
and $\widetilde{\psi}$ is $[-\varphi\circ\vartheta^{-1} ,\vec{e_3}]$-maximal.

\

Finally, from \eqref{integrabilidad} and \eqref{phiL3}, we have 
\begin{align}
& \ll\widetilde{\psi}_x, \widetilde{N}_x\gg= \frac{ \mathrm{e}^{\varphi(u)}}{W} u_{xx} = - \mathrm{e}^{\varphi(u)}<\psi_x,N_x>,\\
& \ll\widetilde{\psi}_x, \widetilde{N}_y\gg= \frac{ \mathrm{e}^{\varphi(u)}}{W} u_{xy} = - \mathrm{e}^{\varphi(u)}<\psi_x,N_y>,\\
& \ll\widetilde{\psi}_y, \widetilde{N}_y\gg= \frac{ \mathrm{e}^{\varphi(u)}}{W} u_{yy} = - \mathrm{e}^{\varphi(u)}<\psi_y,N_y>.
\end{align}
Thus, the shapes operators $A$ ($\widetilde{A}$) of   $\psi$ ($\widetilde{\psi}$) are related as follows:
 \begin{equation}
 \widetilde{A} +   \mathrm{e}^{-\varphi(u)} W^2\ A =0,  \label{sff}
\end{equation}
which, together \eqref{cmetric} let us to prove  that \eqref{hh} and \eqref{kk} hold.
\end{proof}
\begin{remark} Observe that $\widetilde{\psi}$ is a graph on the Legendre transform domain $\Omega_{-1}$ of $\phi$.
\end{remark}
\begin{remark}
The correspondence \eqref{correspondencia1} can be given globally as follows
\begin{equation} \label{globalc}
\widetilde{\psi} = \int  \mathrm{e}^{\varphi (<\psi,\vec{e}_3>) } (\vec{e}_3\wedge (d\psi \wedge N) +  <d\psi,\vec{e}_3> \vec{e}_3),
\end{equation}
where $\wedge$ denotes the cross product in $\R^3$. Moreover,  the singularities of $\widetilde{\psi}$ hold  where the angle  function $<\vec{e}_3,N>$ vanishes.  \end{remark}
Arguing as in  Theorem \ref{th1}, we can prove,
\begin{theorem}\label{th2}
Let  $\Omega_{-1}$ be a simply connected planar domain, $\widetilde{\psi}: \Omega_{-1} \longrightarrow \L^3$, $\widetilde{\psi}(x,y)= (x,y,\overline{u})$ be a $[\varphi ,\vec{e_3}]$-maximal  graph in $\L^3 $ , $\overline{\phi}$ be a solution to the system \eqref{lintegra} and $\vartheta$ be a primitive of $\mathrm{e}^{\varphi }$. Then the immersion given by
\begin{equation}
\psi:= (\overline{\phi}_x,\overline{\phi}_y, \vartheta(\overline{u} )),\label{correspondencia2}
\end{equation}
 is  a $[-\varphi\circ \vartheta^{-1}, \vec{e}_3]$-minimal graph  in $\R^3$ on the Legendre transform domain of $\overline{\phi}$,  whose  induce metric, mean curvature $H$ and Gauss curvature $K$ satisfy
\begin{align}
& g:= \frac{\mathrm{e}^{ 2\varphi(\overline{u})}}{\overline{W}^2} \overline{g},\\
&   H + \mathrm{e}^{-\varphi(\overline{u})} \ \overline{W}^2\  \overline{H} = 0,\label{Lhh}\\
&K  + \mathrm{e}^{-2\varphi(\overline{u})} \ \overline{W}^4\  \overline{K} = 0\label{Lkk},
\end{align}
where $\overline{W}= \sqrt{1- \overline{u}_x^2-\overline{u}_y^2}$ and  $\overline{g}$,  $\overline{H}$ and $\overline{K}$ are the induced metric, the mean curvature and the Gauss curvature of the spacelike graph of $\overline{u}$.
\end{theorem}
\begin{remark}
The correspondence \eqref{correspondencia2} also writes as follows
\begin{equation} \label{lglobalc}
\psi = \int  \mathrm{e}^{\varphi(\ll\widetilde{\psi},\vec{e}_3\gg)} \left(\vec{e}_3\wedge_{\L^3} (d\widetilde{\psi}\wedge_{\L^3}  \overline{N}) - \ll d\widetilde{\psi},\vec{e}_3\gg\vec{e}_3\right),
\end{equation}
where $\wedge_{\L^3}$ denotes the cross product in $\L^3$ and $\overline{N}$ is the Gauss map of $\widetilde{\psi}$. The singular points of 
  $\psi$ hold  where the angle  function $\ll \vec{e}_3,\overline{N}\gg$ vanishes. \end{remark}
  \begin{definition}{\rm If $\psi$ and $\widetilde{\psi}$ are related as in either \eqref{globalc} or \eqref{lglobalc} we say that } $(\psi,\widetilde{\psi})$ is a Calabi-pair.
  \end{definition}
\begin{corollary} Let $(\psi,\widetilde{\psi})$ be a Calabi-pair.
\begin{itemize}
\item If  $\psi$ is a  translating solitons in $\R^3$ then $\widetilde{\psi}$ is a singular $(-1)$-maximal spacelike surface in $\L^3$, see Figure \ref{cp1}.
\item If  $\widetilde{\psi}$   is a  translating solitons in $\L^3$ then $\psi$ is a singular $(-1)$-minimal  surface in $\R^3$, see Figure \ref{cp2}.
\end{itemize}
\end{corollary}
\begin{figure}[htb]
\begin{center}
\includegraphics[width=0.35\linewidth]{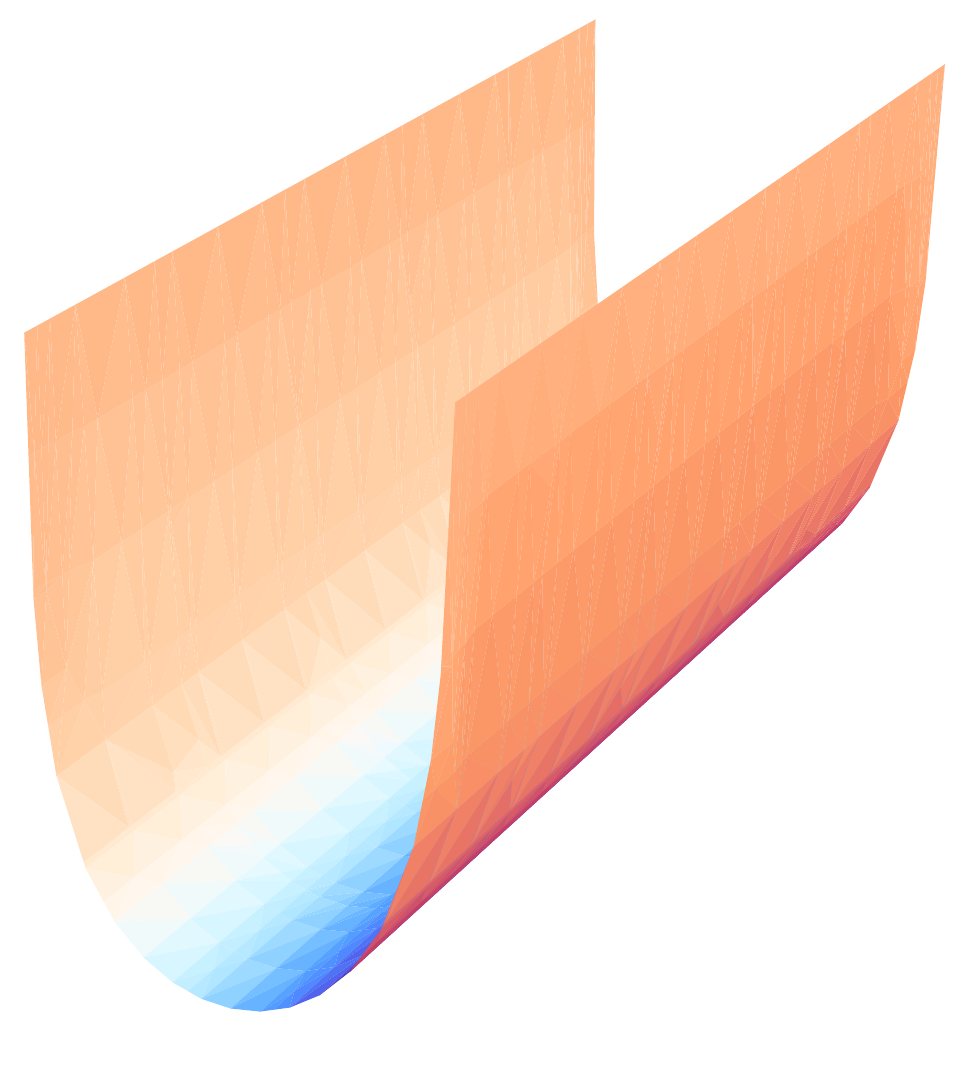}
\includegraphics[width=0.6\linewidth]{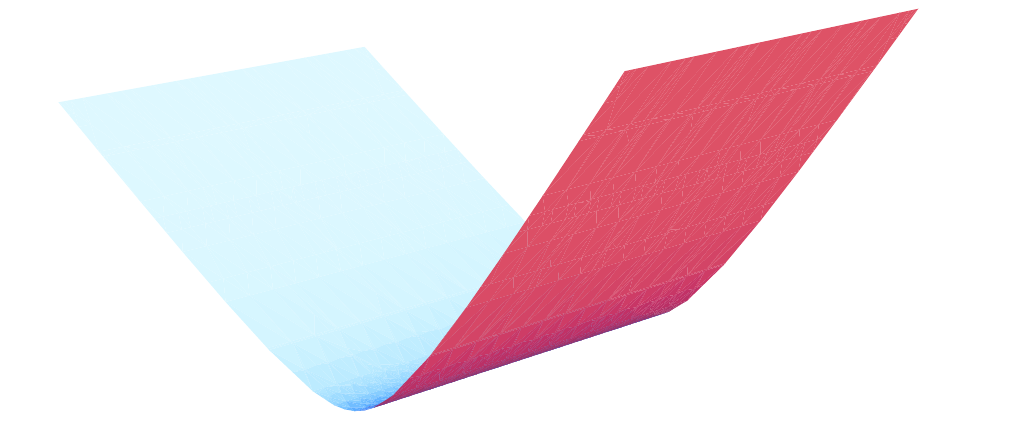}
\end{center}
\caption{Soliton  in $\R^3$ and its corresponding singular (-1)-maximal surface in $\L^3$.}\label{cp1}
\end{figure}
%\begin{corollary} There is a one to one correspondence between singular $(-1)$-minimal surfaces in $\R^3$  and translating solitons in $\L^3$.
%\end{corollary}
\begin{figure}[htb]
\begin{center}
\includegraphics[width=0.45\linewidth]{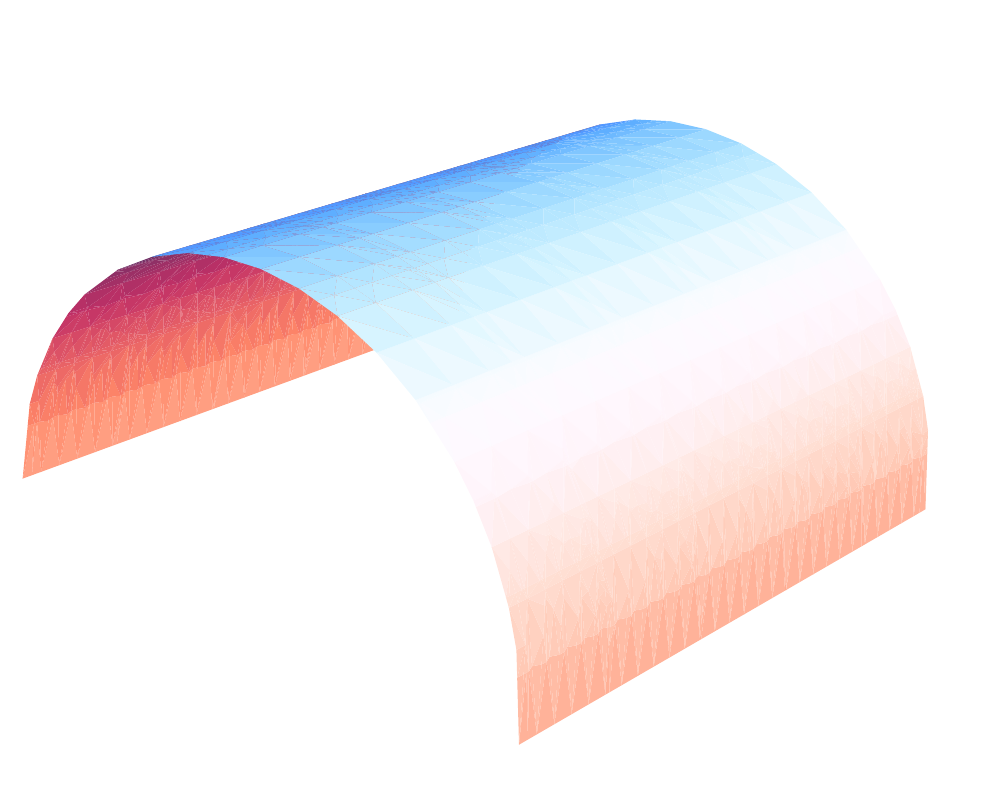}
\includegraphics[width=0.45\linewidth]{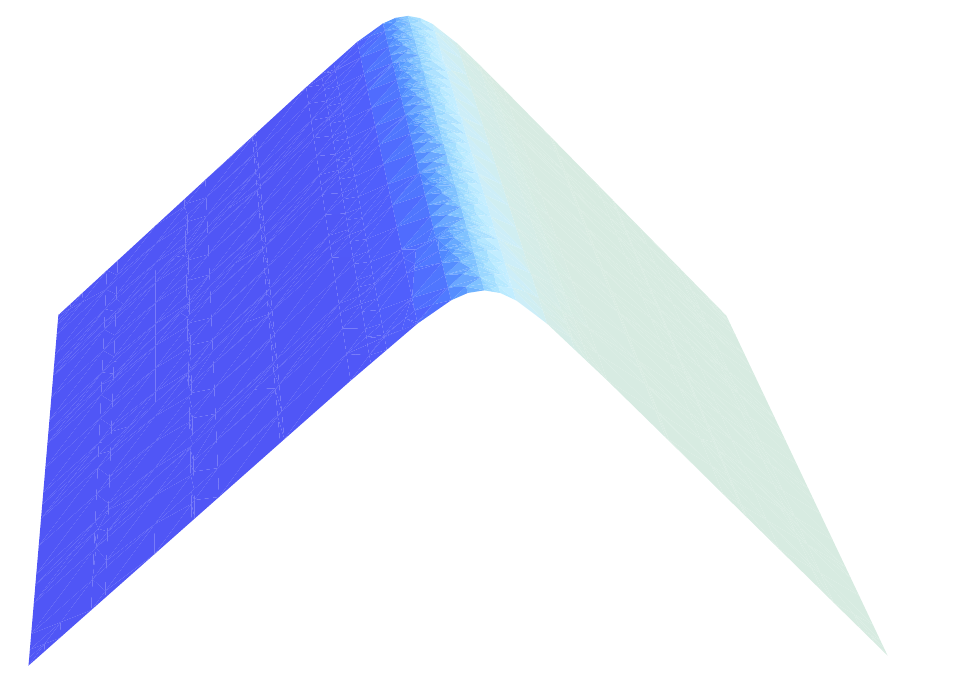}
\end{center}
\caption{Cupola in $\R^3$ and  its corresponding translating soliton in $\L^3$.}\label{cp2}
\end{figure}
\begin{remark}
{\rm  Observe that if one of the surfaces in a Calabi-pair is convex (respectively, of vanishing  Gauss curvature), then both are convex (respectively, of vanishing  Gauss curvature).}
\end{remark}
\section{Applications}
In this section we show some applications to the description   and study of new examples.
\subsection{Radially symmetric  is a preserved condition}
The elliptic equations \eqref{fe} and \eqref{Lfe} have radially symmetric solutions $u (r)$ and $\overline{u}(r)$,  $r=\sqrt{x^2+y^2}$, respectively, if and only if the following ODE equations are satisfied,
\begin{align}
&\frac{u''}{1+u\,'^2} = \dot{\varphi}(u) - \frac{u\,'}{r}, \quad %\text{ for $r\in ]a,b[ \subseteq \R^+$},  
\label{rfe} \\
&\frac{\overline{u}''}{1- \overline{u}\,'^2} = -\dot{\varphi}(\overline{u}) - \frac{\overline{u}\,'}{r}. \quad %\text{ for $\overline{r}\in ]a,b[ \subseteq \R^+$}.  
\label{Lrfe}
\end{align}
From these equations we get,
 \begin{align}
&\left(  \frac{\mathrm{e}^{\varphi(u)} r}{\sqrt{1+u\,'^2}}\right) ' =\mathrm{e}^{\varphi(u)}\sqrt{1+u\,'^2},   \label{fe1} \\
&\left(  \frac{\mathrm{e}^{\varphi(\overline{u})} r}{\sqrt{1-\overline{u}\, '^2}}\right) ' =\mathrm{e}^{\varphi(\overline{u})}\sqrt{1-\overline{u}\, '^2}.\label{fe2}
\end{align}
Hence, if we consider $\phi$ and $\overline{\phi}$ the two radial functions 
\begin{align}
& \phi  =  \int\mathrm{e}^{\varphi(u)} r \cos(z)dr,  \qquad  \overline{\phi}  =  \int\mathrm{e}^{\varphi(\overline{u})} r \cosh(z) dr\label{ephi1}
\end{align}
where
\begin{align*}& u\,'(r)=\tan{z}, \qquad  \overline{u}\, '(r)=\tanh{z},\end{align*}
we have,
\begin{align}\phi'' = \frac{ \mathrm{e}^{\varphi(u)}}{\cos(z)}, \qquad \overline{\phi}\ '' = \frac{ \mathrm{e}^{\varphi(\overline{u})}}{\cosh(z)}.\label{ephi2}
\end{align}
From \eqref{ephi1},  \eqref{ephi2} and by a straightforward computation, we can prove that  $\phi$ and $\overline{\phi}$  are, respectively, solutions of  the differential systems \eqref{integrabilidad} and \eqref{lintegra}. 
Thus, from Theorem \ref{th1} and Theorem \ref{th2} we have,
\begin{proposition}\label{p1}\noindent
\begin{itemize}
\item Let  $\psi := ( r \cos{t}, r \sin{t}, u(r))$ be a revolution $[\varphi,\vec{e}_3]$-minimal surface in $\R^3$. Then, the corresponding $[-\varphi\circ \vartheta^{-1} ,\vec{e_3}]$-maximal  spacelike surface in \eqref{correspondencia1} is the rotational surface of elliptic type given by,
$$  \widetilde{\psi} := \left(\dot{\vartheta}(u)\, r \cos(z)\cos{t}, \dot{\vartheta}(u)\, r \cos(z) \sin{t}, \vartheta(u)\right), $$ 
\item Let  $\widetilde{\psi} := ( r \cos{t}, r \sin{t}, \overline{u}(r))$ be a revolution $[ \overline{\varphi},\vec{e}_3]$-maximal spacelike surface of elliptic type in $\L^3$. Then, the corresponding $[- \varphi\circ \vartheta^{-1} ,\vec{e_3}]$-minimal   surface in \eqref{correspondencia2} is the rotational surface given by,
$$ \psi := \left( \dot{\vartheta}(\overline{u})\, r \cosh(z) \cos{t},  \dot{\vartheta}(\overline{u})\,  r \cosh(z) \sin{t}, \vartheta( \overline{u})\right),$$ 
where
\begin{align*}& u\,'(r)=\tan{z}, \qquad  \overline{u}\, '(r)=\tanh{z},\end{align*}
\end{itemize}
\end{proposition}
\subsection{Translating solitons and singular $\alpha$-maximal surfaces with  $\alpha>1$.} 
The behaviour of a rotational singular $\beta$-minimal surface $\psi$ in $\R^3$, $\beta\leq -1$, was described in \cite[Theorem 8]{R1}. In fact, if 
$$\gamma(s) = (x(s),0,u(s)), \quad s\in [0,L],$$
is the arc-length parametrized generating curve of $\psi$, there are only two possibilities:

\

\noindent {\sc Case I: $\gamma$ intersects the rotation axis}

\noindent In this case, although the ordinary differential equation associated to  problem has a singularity at $x=0$, see \eqref{rfe}, the curve  $\gamma$ is the graph of a symmetric and concave function $u(x)$, $u\in {\cal C}^2(]-R,R[)$ and  meets orthogonally the rotation axis  and  to the  $x$-axis, see Figure \ref{f3}. In particular,
\begin{figure}[htb] 
\begin{center}
\includegraphics[width=0.55\linewidth]{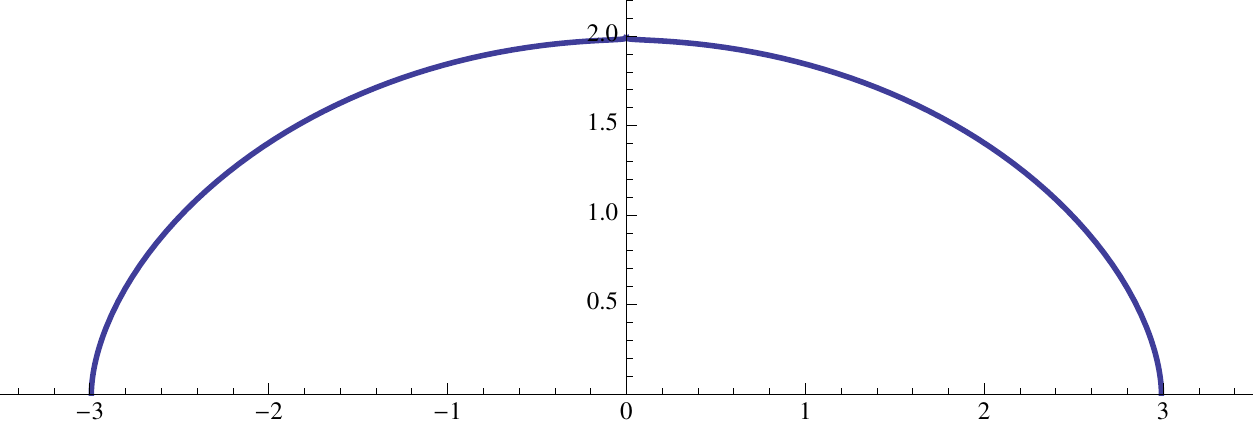}
\includegraphics[width=0.4\linewidth]{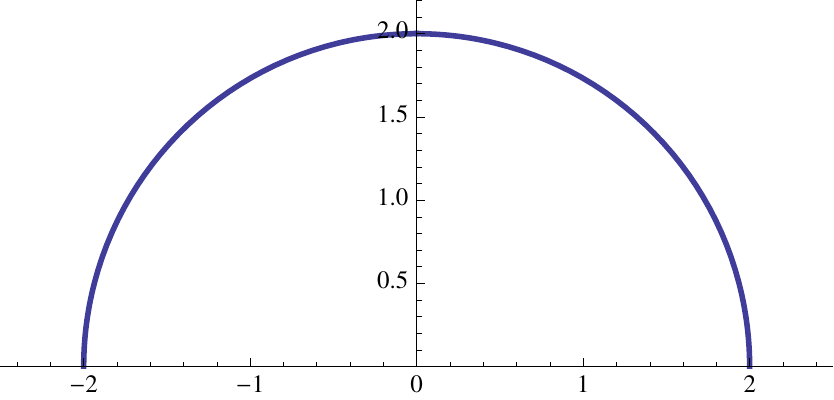}
\end{center}
\caption{generating curve $\gamma$ for $\beta=-1$ and for $\beta=-2$.} \label{f3}
\end{figure}
\begin{description}
\item[(a)] $u(0)=u(L)=0$, $x(0)=-R$, $x(L)=R$, $x'(0)=x'(L)=0$.
\item[(b)]  $x(L/2)=0$, $u(L/2)>0$, $u'(L/2)=0$.
\end{description}
\begin{theorem}[Existence of spacelike singular $\alpha$-maximal bowls of elliptic type]\label{texistencia1}\noindent
\begin{itemize}
\item[(i)] There exists a rotationally symmetric,  entire, smooth, strictly convex spacelike translating soliton (unique up to translation) and of linear growth, Figure \ref{f4} left (see \cite{QD,JIAN} for a proof based on ODE theory). 
 \item[(ii)] For any  $\alpha>1$,  there is a  rotationally symmetric,  entire, smooth, strictly convex spacelike singular  $\alpha$-maximal graph (unique up to  homothety) and of linear growth, Figure \ref{f4} right.
 \end{itemize}
These examples will be called {\sl singular $\alpha$-maximal bowls} of elliptic type.
\end{theorem}
\begin{figure}[htb] 
\begin{center}
\includegraphics[width=0.4\linewidth]{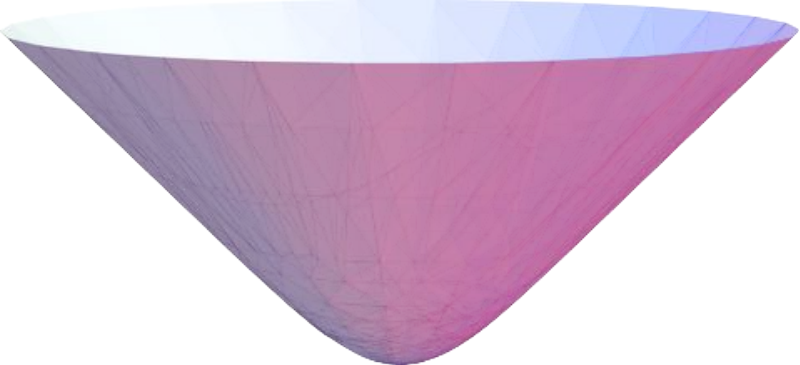}\ \ \ 
\includegraphics[width=0.4\linewidth]{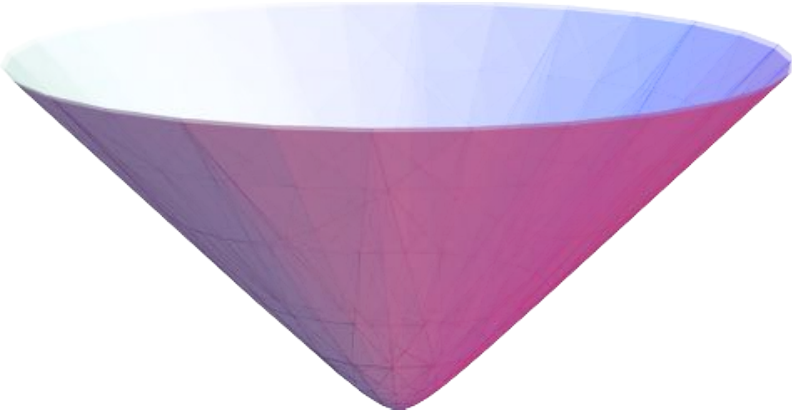}
\end{center}
\caption{Translating soliton and singular $2$-maximal bowl of elliptic type in $\L^3$} \label{f4}
\end{figure}
\begin{proof}
As $\gamma$ is the generating curve of $\psi$, from \eqref{rfe}, \eqref{ephi1} and \eqref{ephi2}, we have
\begin{align}
& x'(s)=\cos{z(s)}\quad  u'(s)=\sin{z(s)}, \label{fgeq1}\\
&  z'(s)= \frac{\beta \cos{z(s)}}{u(s)} - \frac{\sin{z(s)}}{x(s)},\label{fgeq2}
\end{align} 
From  Proposition \ref{p1}, the curve $\widetilde{\gamma}= (\lambda, 0,\vartheta (u) )$ with 
\begin{align}
& \lambda = u^\beta x \cos{z}, \quad \frac{d\vartheta}{du}=\dot{\vartheta}(u) = u^\beta,\label{curve}
\end{align}
is the generating curve of either a revolution spacelike singular $\frac{\beta}{\beta+1}$-maximal surface  of elliptic type with $\beta<-1$  or a translating soliton  with $\beta=-1$ in $\L^3$, $\widetilde{\psi}$ given by
$$ \widetilde{\psi}:=(\lambda \cos (t), \lambda \sin (t), \vartheta(u)).$$
From (a), (b),  \eqref{fe1}, \eqref{fgeq1},  \eqref{fgeq2} and \eqref{curve}, we have,
\begin{align*}
& \lim_{s\rightarrow 0,L}\frac{d\lambda}{ds} = \lim_{s\rightarrow 0,L}u^\beta =  \infty,\qquad   \lim_{s\rightarrow 0,L}\frac{d\vartheta}{ds} = \lim_{s\rightarrow 0,L}u^\beta \sin{z} =  -\infty 
\end{align*}
Hence, $\lambda(s)$ increases in $]0,L[$ from $-\infty$ to $\infty$ and  $$\lim_{s\rightarrow 0,L}\vartheta(s)=-\infty.$$  Moreover,
\begin{align*}
& \lim_{s\rightarrow L/2}\lambda(s) = 0,\quad   \lim_{s\rightarrow L/2}\vartheta = \frac{u(L/2)^{\beta+1}}{\beta+1},  \ \text{(= $\log{u(L/2)}$)} \quad \text{if $\beta<-1$ ($\beta=-1$)}.
\end{align*}
Finally, from \eqref{fgeq1},  \eqref{fgeq2} and \eqref{curve}, we  have also,
\begin{align}
 \frac{d\vartheta}{d\lambda} = \sin{z},\quad  \frac{d^2\vartheta}{d\lambda^2} = \frac{\beta \cos^2{z}}{u^{\beta +1}} - \frac{\sin{2z}}{2 x u^\beta} \leq 0.
\end{align}
So,  $\vartheta$ is  a function of $\lambda$ which satisfies:  $\vartheta\in{\cal C}^2(\R)$,  is  concave, has a maximum at the origin and has a linear growth, see Figure \ref{ff4}.
\begin{figure}[htb] 
\begin{center}
\includegraphics[width=0.4\linewidth]{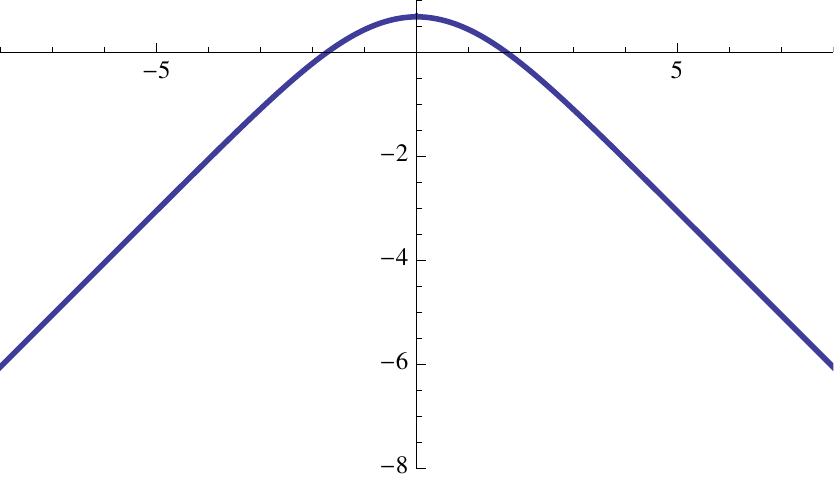} \ \ \ 
\includegraphics[width=0.44\linewidth]{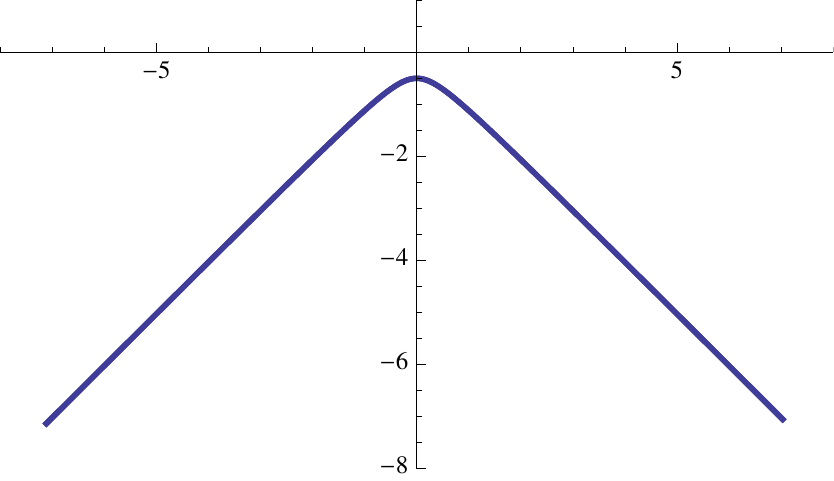}
\end{center}
\caption{generating curve $\widetilde{\gamma}$ for $\beta=-1$ and for $\beta=-2$.} \label{ff4}
\end{figure}

The convex singular $\alpha$-maximal bowl, $\alpha=\frac{\beta}{\beta +1}$, is obtained by the reflexion  in $z=0$ of $\widetilde{\psi}$.

The uniqueness  follows by using the tangency  principle for $[\varphi,\vec{e}_3]$-minimal and spacelike $[\varphi,\vec{e}_3]$-maximal vertical graphs and making into account that a vertical translation moves any translating soliton into another translating soliton and a Euclidean homothety with center at the origin moves any spacelike singular $\alpha$-maximal surface into another spacelike singular $\alpha$-maximal surface.
\end{proof} 
Consider  either $f(\overline{u})\equiv 1$ or $f( \overline{u}) = \alpha/\overline{u}$, for some $ \alpha>1$. Then the above theorem gives the following existence result:
\begin{corollary}
 For each positive real number $a$, the  initial value problem
\begin{align}
& \frac{\overline{u}''}{1- \overline{u}'^2} = f(\overline{u})- \frac{\overline{u}\,'}{r}, \qquad r\in ]0,\infty[ \label{ep1},\\
& \overline{u}(0)= a>0, \qquad \overline{u}\,'(0) =0.\label{ep2}
\end{align}
has an unique convex solution  $\overline{u}\in {\cal C}^2[0,\infty[$ such that 
$$ \lim_{r\rightarrow\infty} \overline{u}\,' = 1.$$

\end{corollary}

\

\noindent {\sc Case II: $\gamma$ does not intersects the rotation axis}

\noindent In this case, $\gamma$ is embedded, has a winglike-shape and  intersects orthogonally the $x$-axis in two points, see Figure \ref{f5}, that is,
 \begin{description}
 \item[(c)] $u(0)=u(L)=0$, $x\,'(0)=x\,'(L)=0$.
 \item[(d)] There exists $s_1\in ]0,L[$, such that  $x\,'(s_1)=0$, $u(s_1)>0$ and $u$ is concave in $[s_1,L]$.
 \end{description}
 \begin{figure}[htb]
\begin{center}
\includegraphics[width=0.48\linewidth]{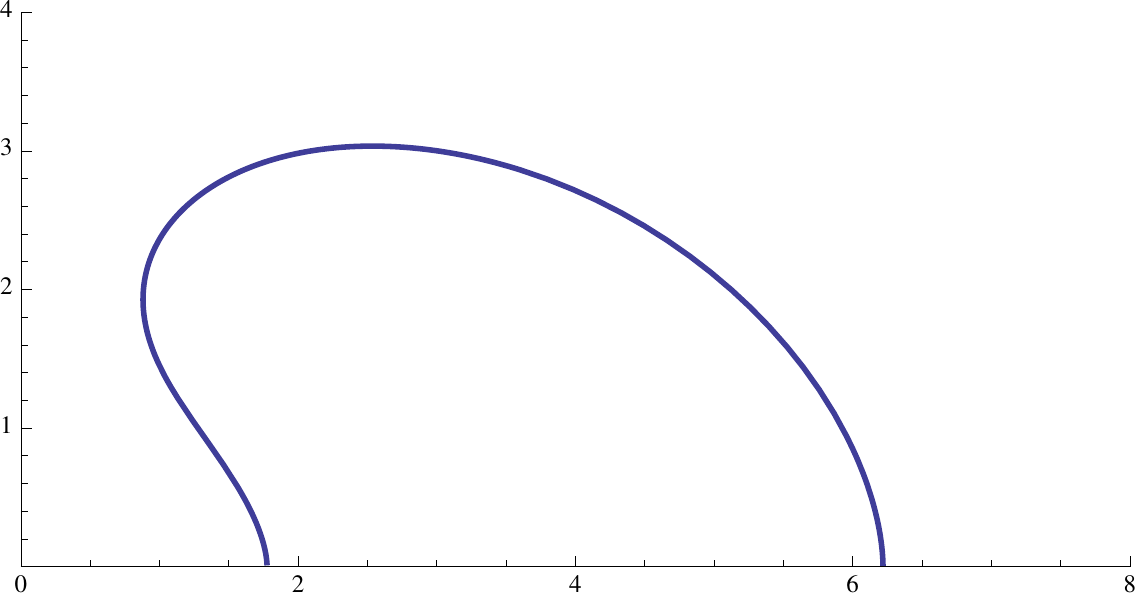} \ \  
\includegraphics[width=0.41\linewidth]{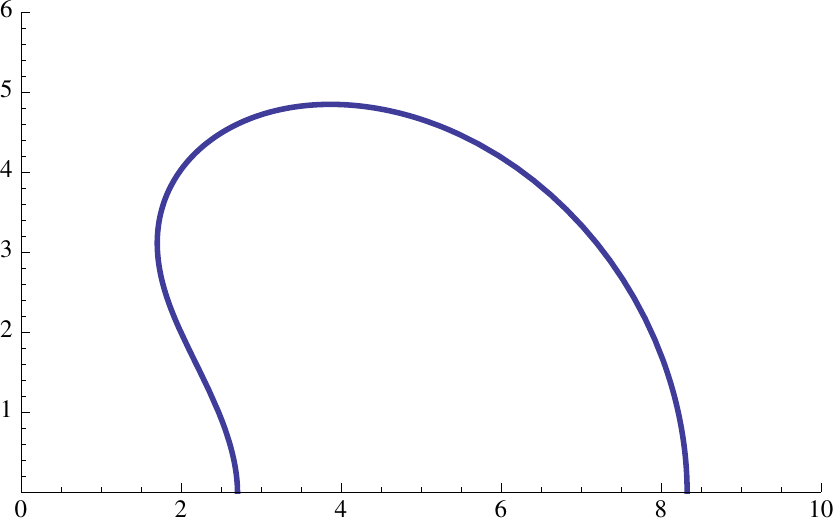}
\end{center}
\caption{Generating curves $\gamma$ for $\beta=-1$ and for $\beta =-1.5$}\label{f5}
\end{figure}

\begin{theorem}[Existence of  winglike spacelike singular $\alpha$-maximal surfaces] \label{texistencia2}

\noindent For any $\alpha >1$, there exist, up to an homothety (translation), two entire spacelike singular $\alpha$-maximal  graphs (translating solitons) in $\L^3$ with linear growth and with an  isolated singularity at the origin which is asymptotics to the light cone.

\

\noindent {\rm This kind of examples will be called either {\sl winglike solitons}  or {\sl  winglike singular $\alpha$-maximal surfaces}, see Figure \ref{fig6}}.
\end{theorem}
\begin{figure}[htb]
\begin{center}
\includegraphics[width=0.4\linewidth]{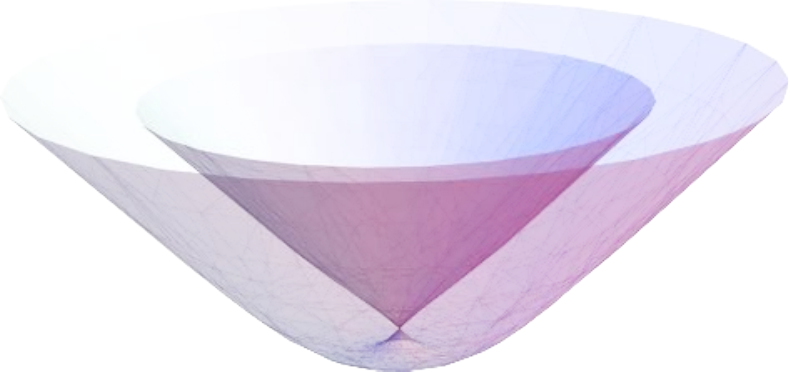} \ \ \ 
\includegraphics[width=0.4\linewidth]{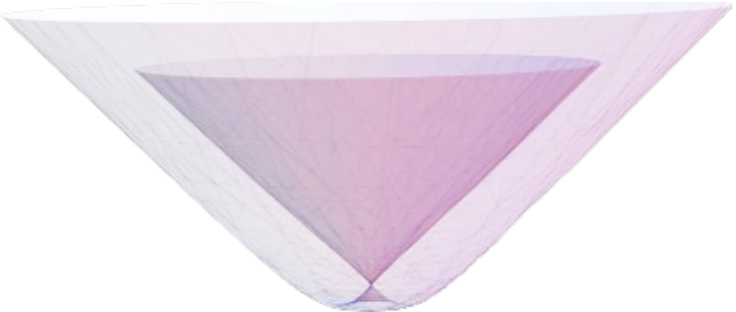}
\end{center}
\caption{winglike soliton and singular $3$-maximal winglike in $\L^3$} \label{fig6}
\end{figure}
\begin{proof}
As in the proof of Theorem \ref{texistencia1}, the curve $\widetilde{\gamma}= (\lambda, 0,\vartheta(u) )$ with 
\begin{align}
& \lambda = u^\beta x \cos{z}, \quad \dot{\vartheta}(u) =u^\beta, \label{curve1}
\end{align}
is the generating curve of a revolution spacelike $\frac{\beta}{\beta+1}$-maximal surface of elliptic type (respectively, translating soliton)  with  $\beta<-1$ (respectively, with $\beta=-1$) in $\L^3$, 
$\widetilde{\psi}$ given by
$$ \widetilde{\psi}:=(\lambda \cos (t), \lambda \sin (t), \vartheta(u)),$$ 
which verifies,
\begin{align*}
&\frac{d\lambda}{ds} = u^\beta ,\qquad  \frac{d\vartheta}{ds} = u^\beta \sin{z} 
\end{align*}
Hence, $  \lim_{s\rightarrow 0,L} \frac{d\lambda}{ds}=\infty $, $  \lim_{s\rightarrow 0,L} \frac{d\vartheta}{ds}=-\infty$ and $\lambda(s)$ increases in $]0,L[$ from $-\infty$ to $\infty$ and  $$\lim_{s\rightarrow 0,L}\vartheta(s)=-\infty.$$  
On the other hand, 
\begin{align*}
& \lim_{s\rightarrow s_1}\lambda(s) = 0,\quad   \lim_{s\rightarrow  s_1}\vartheta = \frac{u(s_1)^{\beta+1}}{\beta+1},  \ \text{(= $\log{u(s_1)}$)} \quad \text{if $\beta<-1$ ($\beta=1$)}.\\
\end{align*}
and
$$ \frac{d\vartheta}{d\lambda} = \sin{z}\qquad \frac{d^2\vartheta}{d\lambda} = \frac{\beta \cos^2{z}}{u^{\beta+1}} - \frac{\sin{2z}}{2 xu^\beta}. $$ 
Finally, 
$$ \lim_{\lambda\rightarrow -\infty,\infty} \frac{d\vartheta}{d\lambda} =-1, \quad  \lim_{\lambda \rightarrow 0 } \frac{d\vartheta}{d\lambda}=1,$$ which says that the isolated singularity  is asymptotic to the light cone. Moreover, as a function of $\lambda$,  $\vartheta$ is of linear growth (see Figure \ref{ff5}). Uniqueness follows by applying  uniquiness of solution of \eqref{fgeq1} and \eqref{fgeq2} with the same initial conditions at $s_1$ and having in mind that a vertical translation moves any translating soliton into another translating soliton and a Euclidean homothety with center at the origin moves any spacelike singular $\alpha$-maximal surface into another spacelike singular $\alpha$-maximal surface. 

 In  Figure \ref{fig6} we have  pictures of the   rotational surfaces with generating curve $ (\lambda,0, -\vartheta(u))$ for $\beta=-1$ and for $\beta=-1.5$.
\end{proof}
\begin{figure}[htb] 
\begin{center}
\includegraphics[width=0.4\linewidth]{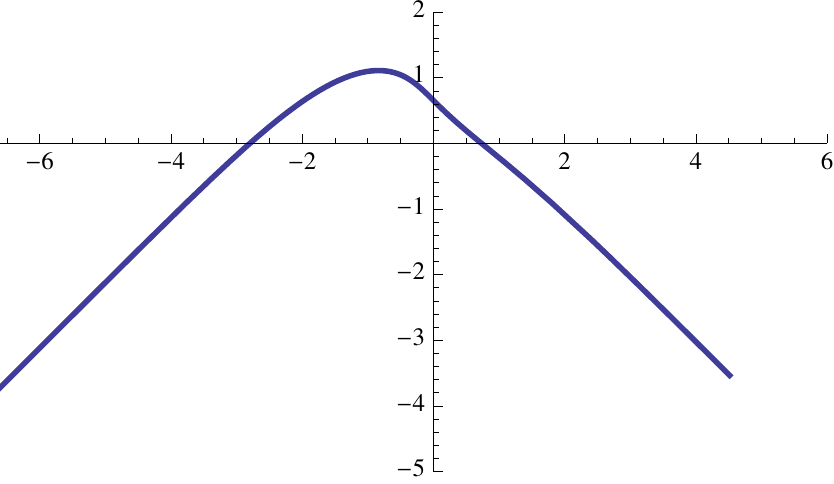} \ \ \ 
\includegraphics[width=0.44\linewidth]{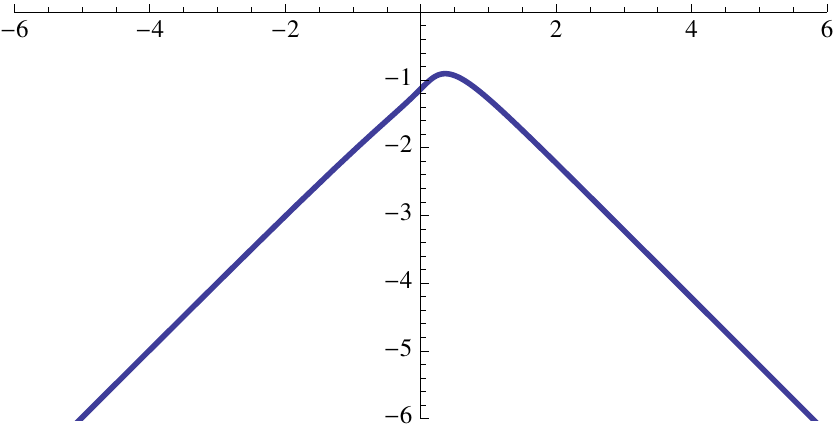}
\end{center}
\caption{generating curve $\widetilde{\gamma}$ with $\beta=-1$ and $\beta=-1.5$.} \label{ff5}
\end{figure}

{\rm If  $f(\overline{u})\equiv 1$ or $f( \overline{u}) = \alpha/\overline{u}$, for some $ \alpha>1$, the Theorem \ref{texistencia2} gives}
\begin{corollary} 
\noindent  For each positive real number $a$,
\begin{align}
&\left\{
\begin{aligned}
&\frac{\overline{u}''}{1- \overline{u}'^2} = f(\overline{u})- \frac{\overline{u}\,'}{r}, \qquad r\in ]0,\infty[ ,\\
&\overline{u}(0)= a>0, \qquad \overline{u}\,'(0) =-1
\end{aligned}\right.\label{pvi1}
\end{align}
has an unique  solution  $\overline{u}\in {\cal C}^2([0,\infty[)$. Moreover, it is strictly convex, has  a minimum in $]0,\infty[$  and satisfies,
$$ \lim_{r\rightarrow\infty} \overline{u}\,' = 1.$$
Moreover,
\begin{align}
\left\{
\begin{aligned}
&\frac{\overline{u}''}{1- \overline{u}'^2} = f(\overline{u})- \frac{\overline{u}\,'}{r}, \qquad r\in ]0,\infty[ ,\\
&\overline{u}(0)= a>0, \qquad \overline{u}\,'(0) =1
\end{aligned}\right.\label{pvi2}
\end{align}
has an unique  solution  $\overline{u}\in {\cal C}^2[0,\infty[$ such that 
$$ \lim_{r\rightarrow\infty} \overline{u}\,' = 1.$$ 
\end{corollary}
\subsection{Rotational surfaces of hyperbolic type }
%From the Calabi's correspondence we give a relationship between the $\Delta$-Wings translating soliton in $\mathbb{R}^{3}$ and $(\alpha-1)$-catenaries in the Minkowski space $\mathbb{L}^{3}$. These examples were studied by D.Hoffman, T.Ilmanen, F.Mart\'in and B.White in \cite{HIMW} where they clasified the translating soliton graphs. In fact, the $\Delta$-Wings are the unique complete translator $u^{a}:(-a,a)\rightarrow\mathbb{R}$ that is not a tilted Grim reaper where $a>\pi/2$. 
%\\
%
%The $\Delta$-Wings are examples of translating soliton graphs verifying the following properties:
%\begin{enumerate}
%\item $u^{a}(x,y)=u^{a}(-x,y)=u^{a}(x,-y)$.
%\item The Gauss curvature of the graph is everywhere positive.
%\item $\text{sup}(a-\vert x\vert)\sqrt{1+\vert\nabla u^{a}\vert^{2}}\leq C(A)$ where $C(A)<\infty$ and $\nabla$ is the gradient in $\mathbb{R}^{3}$ for all $a\leq A, L\geq 1$ and $(x,y)\in [-a,a]\times [-L,L]$.
%\end{enumerate}
%\newpage

In this section, we study rotational $\alpha$-maximal singular surfaces of hyperbolic type in  $\L^3$.
%and use the  Calabi's correspondence to give new examples of $\alpha$-minimal singular surfaces in $\R^3$.
% The vertical graphs $f:(-a,a)\rightarrow\mathbb{R}^{+}$ are called ($\alpha+1)$-catenaries if they satisfy the following differential equation,
%\begin{equation}
%\label{catenaria}
%\frac{f''(s)}{1-f'(s)^{2}}+\frac{1+\alpha}{f(s)}=0 \text{ with } \vert f'(s)\vert<1,
%\end{equation}
%or equivalently,
%\begin{equation}
%\label{catenequ}
%\sqrt{1-f'(s)^{2}}=kf^{\alpha+1}(s) \text{ for some constant } k\in\mathbb{R}^{+}_{0}.
%\end{equation}

These surfaces are invariant by  the 1-parameter group of hyperbolic rotations of the Lorentz group which  fix the $\vec{e}_{1}$ spacelike direction,
%$$\mathcal{G}=\{ \left( \begin{array}{ccc}
%1 & 0 & 0 \\
%0 & \text{cosh}(t) & \text{sinh}(t) \\
%0 & \text{sinh}(t) & \text{cosh}(t) \end{array} \right) : t\in\mathbb{R}\}. $$
A such surface with generating curve the arc-length parametrized  spacelike curve $\widetilde{\gamma}=(x(s), 0,u(s))$, $u>0$, $s\in I\subseteq \R$ is given by,
\begin{equation}
\label{para}
\widetilde{\psi}(s,t)=\left(x(s) ,u(s)\, \sinh(t) ,u(s)\, \cosh(t)\right), \quad (s,t)\in I\times \R,
\end{equation}
with  
\begin{align}
\label{vectortangente}
&x'(s)=\cosh(z(s)), & u'(s)=\sinh(z(s)).
\end{align}

The Gauss map of  $\widetilde{\psi}$ writes as
\begin{equation}\label{Gaussmap}
\widetilde{N}(s,t)=\left(\sinh(z(s)),\cosh(z(s))\sinh(t),\cosh(z(s))\cosh(t) \right),
\end{equation}
and we have,
\begin{equation}
\label{coefi}
\begin{array}{ll}
\ll \widetilde{\psi}_s,\widetilde{\psi}_s \gg=1, &\ll \widetilde{\psi}_s,\widetilde{N}_s\gg= z'(s), \\
\ll \widetilde{\psi}_t,\widetilde{\psi}_t \gg=u(s)^{2}, &\ll \widetilde{\psi}_t,\widetilde{N}_t\gg=u(s) \cosh (z(s)),  \\
\ll \widetilde{\psi}_s,\widetilde{\psi}_t \gg=0, &\ll \widetilde{\psi}_s,\widetilde{N}_t\gg=0. \\
\end{array}
\end{equation}
From (\ref{coefi}), the mean curvature field of $\widetilde{\psi}$ is given by
\begin{equation}
\label{meancurvature}
{\widetilde{\textbf{H}}}=\left(\frac{u\ z' +\cosh (z)}{u} \right)\widetilde{N}= \left(\frac{u\ \mathcal{K} +\cosh (z)}{u} \right)\widetilde{N},
\end{equation}
where $\mathcal{K}$ is the curvature of $\widetilde{\gamma}$.

Consequently, from \eqref{fmaximal}, \eqref{para} and \eqref{Gaussmap}, the spacelike surface $\widetilde{\psi}$ is a singular $\alpha$-maximal spacelike surface if and only if 
\begin{equation}
\label{equz}
\frac{z'(s)}{\text{cosh}(z(s))}+\frac{1+\alpha}{u(s)}=0,
\end{equation}
or equivalently,
\begin{equation}
\label{catenequ}
\cosh (z) \ u^{\alpha +1} = k, \quad \text{ for some positive constant $k$}.
\end{equation}
Moreover, from   \eqref{equz} and (\ref{catenequ}),  the curvature of $\widetilde{\gamma}$ satisfies
\begin{equation}
\label{curvaturecurve}
\mathcal{K}=-(1+\alpha)k \ u^{-\alpha-2}.
\end{equation}

\

From the spacelike condition, \eqref{vectortangente}, \eqref{equz} and \eqref{curvaturecurve}, if $1+\alpha >0$ (respectively, $1+\alpha <0$), the generating curve $\widetilde{\gamma}$ is the graph of a strictly concave (respectively, convex) function $u(x)$ solution of the following ordinary differential equation,
\begin{align}
&\frac{d^2u}{dx^2}= -\frac{\alpha +1}{u}(1-(\frac{du}{dx})^2),\label{cequ}
\end{align}
or equivalently, 
\begin{align}
&\frac{du}{dx} = \tanh (z), & \frac{dz}{dx}=-\frac{1+\alpha}{u}.\label{asystem1}
\end{align}
As a first integrate of this system is given by  \eqref{catenequ}, if $1+\alpha\neq 0$, there exists a unique $x_0\in\R$ such us $u'(x_0)=0$ (see  Figure \ref{f8} for a representation of the trajectories of \eqref{asystem1}). So, up to translation in the $x$-axis, we may assume that $x_0=0$ and consider the solutions to \eqref{asystem1} satisfying
\begin{equation}
u(0)=u_0>0, \qquad \frac{du}{dx}(0) =0. \label{ciniciales}
\end{equation}
By taking $\overline{u}(x)= u(-x)$, we see easily that, if $1+\alpha\neq 0$, a solution to   \eqref{cequ}-\eqref{ciniciales} is even and, from \eqref{catenequ} and \eqref{asystem1}, it is defined in the  interval $]-\Lambda_{u_0},\Lambda_{u_0}[$,  where 
\begin{equation}
\Lambda_{u_0} = |\alpha+1| \int_{0}^\infty \frac{ u_0\ d\tau}{(\cosh \tau)^\frac{1}{\alpha+1}}.\label{Lambda}
\end{equation}

\begin{figure}[htb] 
\begin{center}
\includegraphics[width=0.4\linewidth]{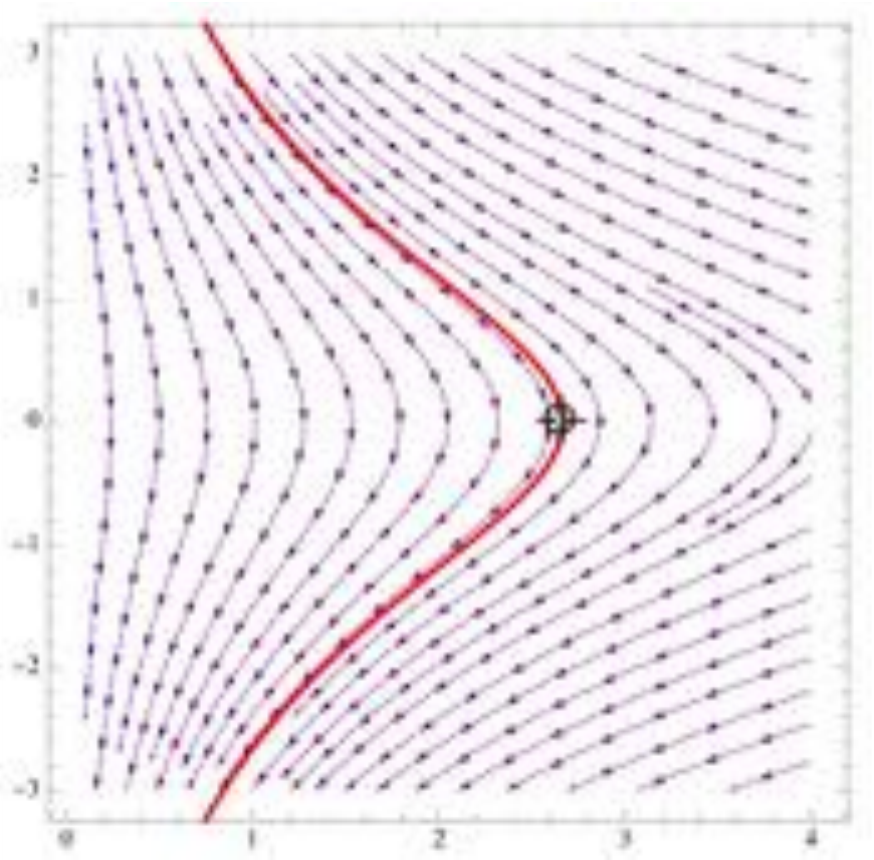} \qquad
\includegraphics[width=0.4\linewidth]{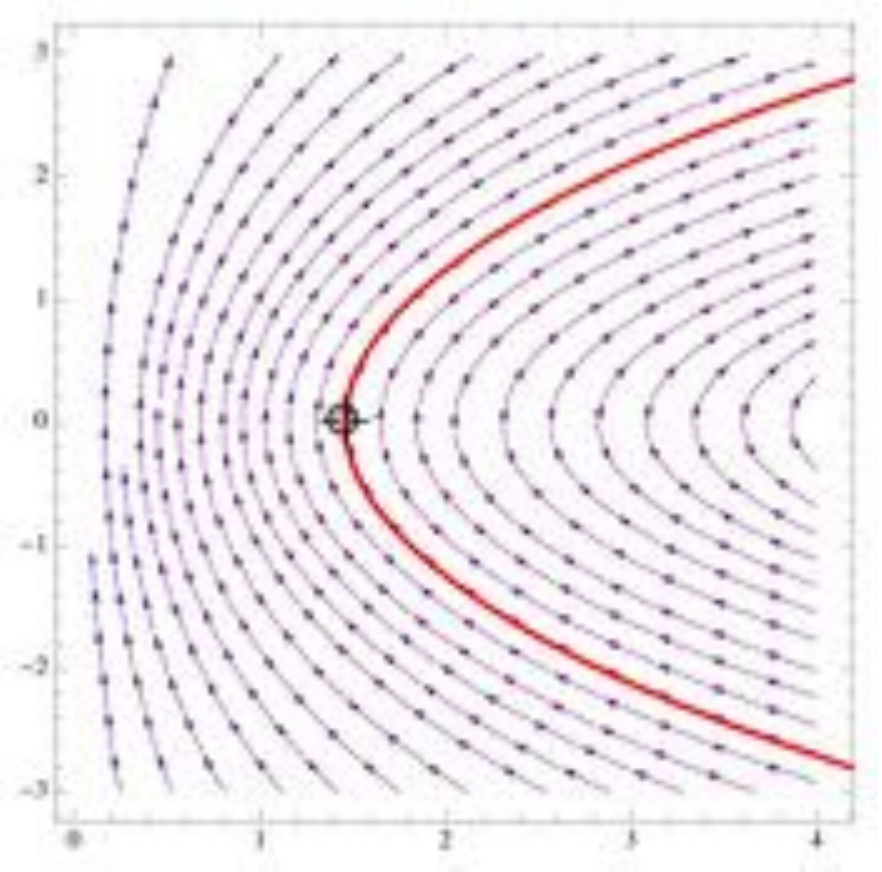}
\end{center}
\caption{Phase portrait of  \eqref{asystem1} for $\alpha = 1$ (left) and  for  $\alpha=-3$ (right).} \label{f8}
\end{figure}
So, from \eqref{catenequ},  \eqref{asystem1} and   \eqref{Lambda}, we have
\begin{proposition} \label{p6} Let $\widetilde{\gamma}=(x,0,u)$ be  the generating curve of a rotational spacelike $\alpha$-maximal surface of hyperbolic type. Then, up to horizontal translation, we have, see figure \ref{f9}
\label{examplesL3}
\begin{itemize}
\item if $1+\alpha>0$,  $\widetilde{\gamma}$ is the graph of a symmetric and strictly concave function $u(x)$ in a bounded interval $-]\Lambda_{u_0},\Lambda_{u_0}[$ which  has a maximum at $0$ and meets (asymptotically to the light cone)  the  $x$-axis in $ \pm\Lambda_{u_0} $,  that is
$$\lim_{x\rightarrow \pm \Lambda_{u_0}}u(x)=0,\quad  \lim_{x\rightarrow \pm \Lambda_{u_0}}\frac{du}{dx}=-1.$$
\item if $1+\alpha<0$,  $\widetilde{\gamma}$   is the graph of a symmetric and strictly convex function $u(x)$ in $]-\infty,\infty[$ which  has a minimum at $0$ and has  linear growth,  in fact 
$$\lim_{x\rightarrow \pm \infty}u(x)=\infty, \quad \lim_{x\rightarrow \pm \infty}\frac{du}{dx}=1.$$
\item if $1+\alpha=0$, then $\widetilde{\gamma}$ is a straight line.
\end{itemize}
\end{proposition}
\begin{figure}[h] 
\begin{center}
\includegraphics[width=0.4\linewidth]{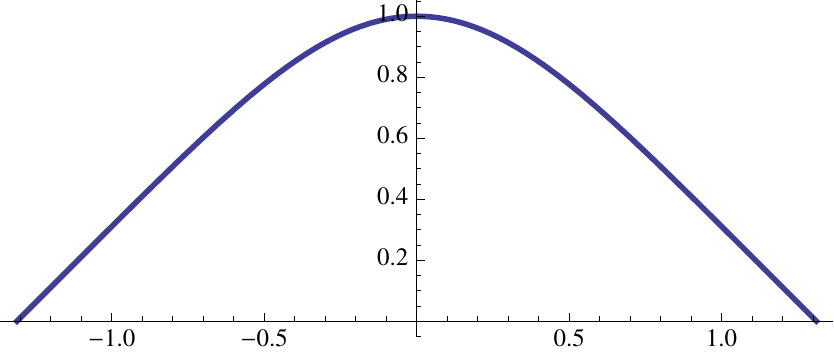} \ \ \ 
\includegraphics[width=0.44\linewidth]{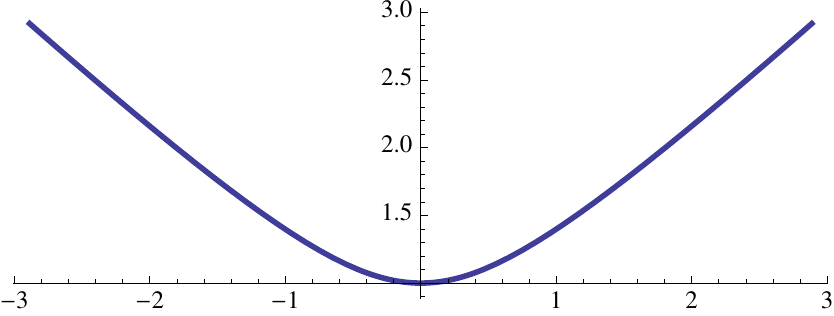}
\end{center}
\caption{Generating curve $\widetilde{\gamma}$ for $\alpha=1$ (left) and for $\alpha =-2$ (right).} \label{f9}
\end{figure}
\begin{remark}\label{r7}
Observe that from \eqref{para}, \eqref{coefi} and \eqref{catenequ}, the Gauss curvature of $\widetilde{\psi}$ is given by $$ K = (\alpha + 1)\frac{k^2 }{u^{2\alpha +4}},$$
and so, see Figure \ref{f10}
\end{remark}
\begin{theorem} \label{Kurv} Let $\widetilde{\psi}$ be  a spacelike surface $\widetilde{\psi}$  a singular $\alpha$-maximal spacelike graph, parametrized as in \eqref{para} with $\widetilde{\gamma}=(x,0,u)$ as in Proposition \ref{p6}, then
\begin{itemize}
\item if $1+\alpha>0$, the Gauss curvature of $\widetilde{\psi}$  is positive (that is, the second fundamental form is semidefinite and nondegenerate). Moreover,  
$\lim_{x\rightarrow \pm\Lambda_0}K = \infty.$  
\item if $1+\alpha<0$, $\widetilde{\psi}$ is an entire graph, strictly convex with a flat behaviour at infinity,   that is, 
$\lim_{x\rightarrow \pm\infty}K =0.$

(These examples will be called singular  {\sl  $\alpha$-maximal bowls of hyperbolic type}.)
\item If $\alpha=-1$, $\widetilde{\psi}$  is flat.
\end{itemize} 
\end{theorem}
\begin{figure}[h] 
\begin{center}
\includegraphics[width=0.27\linewidth]{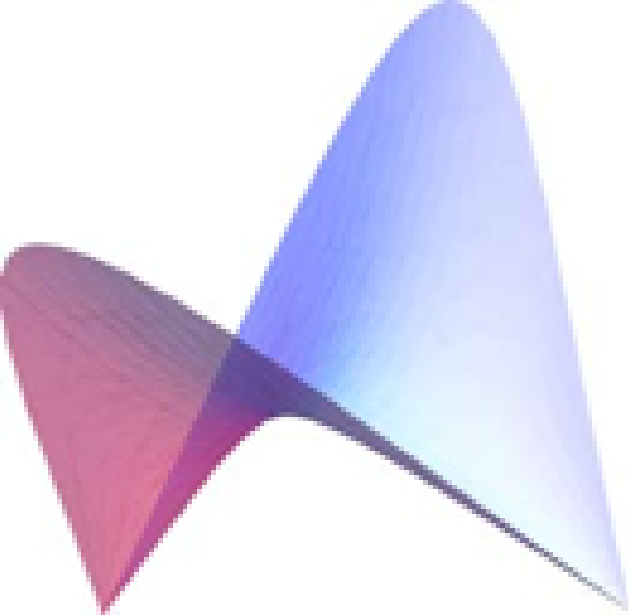} \ 
\includegraphics[width=0.34\linewidth]{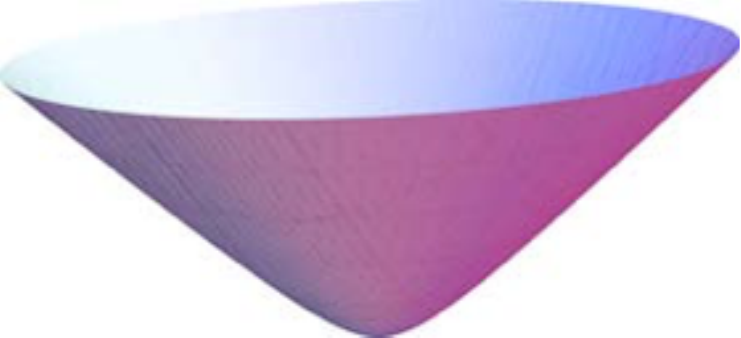} \
\includegraphics[width=0.35\linewidth]{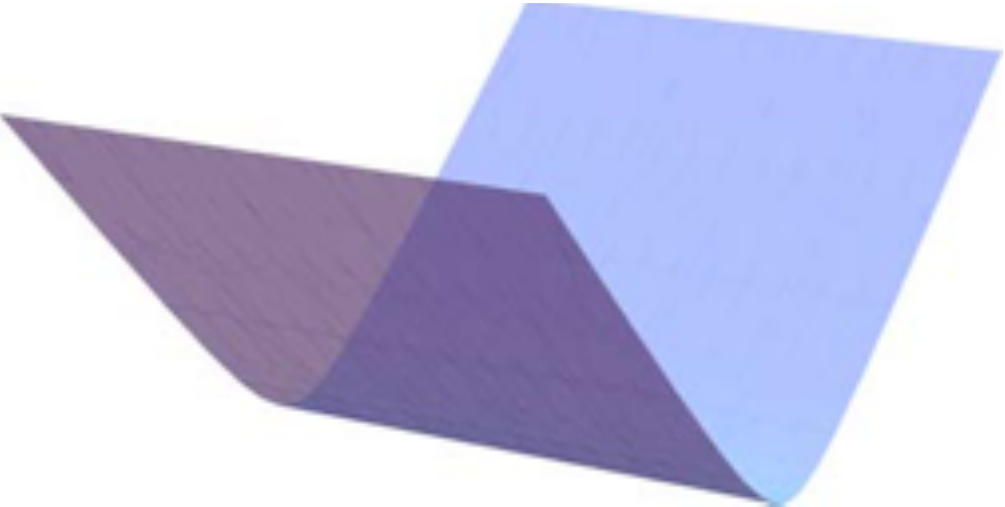}
\end{center}
\caption{Rotational singular $\alpha$-maximal surfaces of hyperbolic type for $\alpha=1$,  $\alpha =-2$ and $\alpha=-1$).} \label{f10}
\end{figure}
\begin{remark} Observe that, from \eqref{catenequ}, the divergent generating curve $\widetilde{\gamma}$ of a $\alpha$-maximal bowl of hyperbolic type,  has infinite length if and only if $$\int_0^\infty u^{\alpha +1} dx = \infty$$ which, because $u$ has a linear growth, only holds when $\alpha + 2\geq 0$. So, a singular $\alpha$-maximal bowl of hyperbolic type is complete  if and only if $-1>\alpha\geq -2$.
\end{remark}

\subsection{ Rotational  singular $\alpha$-maximal spacelike surfaces of hyperbolic type in a Calabi-pair}

%So, taking $\phi_{\tau}$ an isometry of the hyperbolic group given by
%$$\begin{pmatrix}
%1 & 0 & 0 \\
%0 & \text{cosh}(\tau) & \text{sinh}(\tau) \\
%0 & \text{sinh}(\tau) &  \text{cosh}(\tau) 
%\end{pmatrix},$$
%we get that $$\phi_{\tau}\circ X=X(s,\theta+\tau) \emph{ , } d\phi_{\tau}\circ N=N\circ X(s,\theta+\tau).$$ Hence,
%$$\ll d\phi_{\tau}(N),\vec{e}_{3}\gg=\ll \widetilde{N}\circ\phi_{\tau},\vec{e}_{3}\gg=H=\widetilde{H}$$
%and the property of be $(\alpha-1)$-catenary is invariant by hyperbolic group of isometries. Notice that, $\widetilde{H}$ and $\widetilde{N}$ is the mean curvature and Gauss map of the spacelike surface $\widetilde{X}=\phi_{\tau}\circ X$.

%Moreover, $0<\ll N,\vec{e}_{3}\gg $  and the spacelike surface $S$, associated to the parametrization $X$ , encloses a domain of $\mathbb{L}^{3}$ together the upper half space of the plane $\{ z=0\}$. So, the spacelike surface is a vertical graph in $\mathbb{L}^{3}$ and denoting by,
%$$\textbf{x}=\text{sinh}(\theta)f(s) \emph{ and } u(s,\textbf{x})=\sqrt{\textbf{x}^{2}+f(s)^{2}},$$
%we can write the parametrization $X$ as follow,
%$$\widetilde{X}(s,\textbf{x})=(s,\textbf{x},u(s,\textbf{x})) \emph{ , } s\in (-a,a) \emph{ , } \textbf{x}\in (-\delta,\delta).$$
%
%Consequently, 
In this section we study   the Calabi-pairs $(\psi,\widetilde{\psi})$ with   $\widetilde{\psi}$  a  rotational singular $\alpha$-maximal spacelike surfaces of hyperbolic type . 

\

Let $\widetilde{\psi}$ be the rotational singular $\alpha$-maximal spacelike surface of hyperbolic type given by \eqref{para} with Gauss map \eqref{Gaussmap}. If $(\psi,\widetilde{\psi})$ is a Calabi-pair,  then from Theorem \ref{th2} and \eqref{lglobalc}, $\psi$ is either a singular $\frac{\alpha}{\alpha+1}$-minimal surface if $\alpha+1\neq0$ or a translating soliton if $\alpha+1=0$ in $\R^3$ which can be parametrized as
\begin{align}
\label{param}
&\psi(s,t)=\int u(s)^{\alpha}\cosh^{\alpha}t\ \left(\vec{e}_{3}\wedge_{\mathbb{L}^{3}}(d\widetilde{\psi}\wedge_{\mathbb{L}^{3}}\widetilde{N})-\ll d\widetilde{\psi},\vec{e}_{3}\gg\vec{e}_{3}\right).
\end{align}

\

\noindent {\sc Case I: If $\psi$ is a singular $\frac{\alpha}{\alpha+1}$-minimal surfaces}

Then,  from Theorem \ref{th2}, Proposition \ref{p6}, Theorem \ref{Kurv} and  Remark \ref{r7}, we have 
\begin{theorem} Let $(\psi,\widetilde{\psi})$ be  a Calabi-pair such that  $\widetilde{\psi}$ is  the rotational singular $\alpha$-maximal,    $\alpha+1\neq0$, spacelike surface of hyperbolic type given by \eqref{para}. Then $\psi$ is  
\begin{itemize}
\item either a singular $\frac{\alpha}{1+\alpha}$-minimal $\psi$ in $\mathbb{R}^{3}$ whose Gauss curvature is strictly negative if  $1+\alpha>0$.
\item or a singular strictly convex $\frac{\alpha}{1+\alpha}$-minimal $\psi$ in $\mathbb{R}^{3}$ if  $1+\alpha<0$.
\end{itemize}
In both cases a  parametrization of $\psi$ is given by,
\begin{align*}
\psi(x,t)=\left(-\frac{k \ u'(x)}{(\alpha+1)}\text{cosh}^{\alpha+1}(t), k \int\text{cosh}^{\alpha}(t)\, dt,\frac{u^{\alpha+1}(x)\text{cosh}^{\alpha+1}(t)}{\alpha+1} \right),
\end{align*}
where $k=u_0^{\alpha +1}$ and $u(x)$ is the solution of  \eqref{cequ}-\eqref{ciniciales} satisfying the properties describe in Proposition \ref{p6}, see Figure \ref{f12}.
\end{theorem}
\begin{figure}[h] 
\begin{center}
\includegraphics[width=0.3\linewidth]{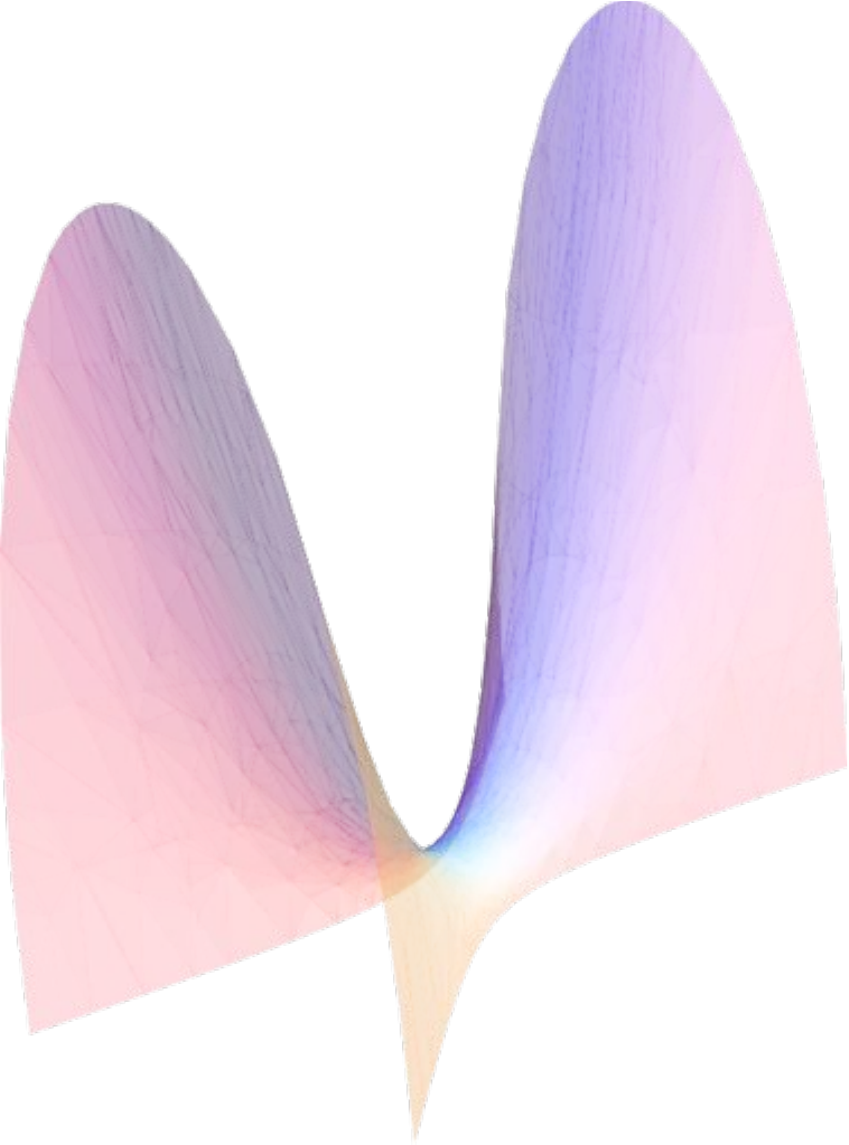} \ \ \ \ \ \
\includegraphics[width=0.5\linewidth]{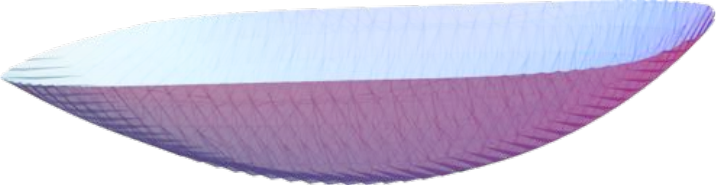} \
\end{center}
\caption{singular $\frac{\alpha}{\alpha+1}$-minimal surfaces  for $\alpha=1$ and for $\alpha =-2$}  \label{f12}
\end{figure}
%\begin{remark}
%When $\alpha=0$, the solution $u$ of (\ref{catenequ}) is given by,
%\begin{equation}
%\label{solacero}
%u(x)=\text{sin}(kx) \emph {where k depending of the initial conditions. }
%\end{equation}
%Moreover, the mean curvature of the correspondent rotational singular $\alpha$-maximal surface of hyperbolic type vanishes, that is, it is a maximal surface in the Minkowski space $\mathbb{L}^{3}$.
%\\
%
%Conquently, applying \eqref{param}, the correspondent minimal surface from the Calabi's correspondence is parametrized by,
%$$\widehat{\Psi}=\left(k\text{cos}(s)\text{cosh}(\theta),-\frac{1}{k}\theta,\text{sin}(ks)\text{cosh}(\theta)  \right).$$
%\begin{corollary} 
%The catenoid is the only minimal surface in $\mathbb{R}^{3}$ from the Calabi's correspondence for maximal surfaces in $\mathbb{L}^{3}$ invariant by the hyperbolic group. 
%\end{corollary}
%\end{remark}

{\sc Case II: If $\psi$ is  a translating solitons ($\alpha=-1$)}
%
%Let $\widetilde{\gamma}$ be a generating curve of a rotational singular $(-1)$-maximal surface $\widetilde{\psi}$ of hyperbolic type.
%\\

From the Proposition \ref{p6}, $\widetilde{\psi}$ could be parametrize (up to translation) by,
\begin{equation}
\label{paramenos1}
\widetilde{\psi}(x,t)=\left(x,(\tanh(z_0) x+u_0)\sinh(t),(\tanh(z_0) x+u_0)\cosh(t)\right), 
\end{equation}
for some  $z_0,x_0\in\mathbb{R}$, where $t\in\R$ and $\tanh(z_0) x+u_0>0$.

\

Consequently, from  \eqref{param} and \eqref{paramenos1} we have 
\begin{theorem}
Let $(\psi,\widetilde{\psi})$ be  a Calabi-pair such that  $\widetilde{\psi}$ is the revolution singular $(-1)$-maximal spacelike surface given by \eqref{paramenos1}. Then, $\psi$ is either the  titled Grim-Reaper  parametrized as
\begin{align*}
 & \psi(y,t)=\left(\frac{y}{\lambda}- \lambda\log(\cosh(t)), 2 \sqrt{1+\lambda^2} \arctan(\tanh(t/2)) , y + \cosh(t))\right),
\end{align*}
if $\lambda=\sinh (z_0)\neq0$, or to the Grim-Reaper parametrized as,
\begin{align*}
&\psi(y,t)=\left(-\frac{y}{u_0},2 \arctan(\tanh(t/2)),\log(u_0\cosh(t))\right),
\end{align*}
if $\lambda=0$ (see Figure \ref{f11}).
\end{theorem}
\begin{figure}[h]
\begin{center}
\includegraphics[width=0.2\linewidth]{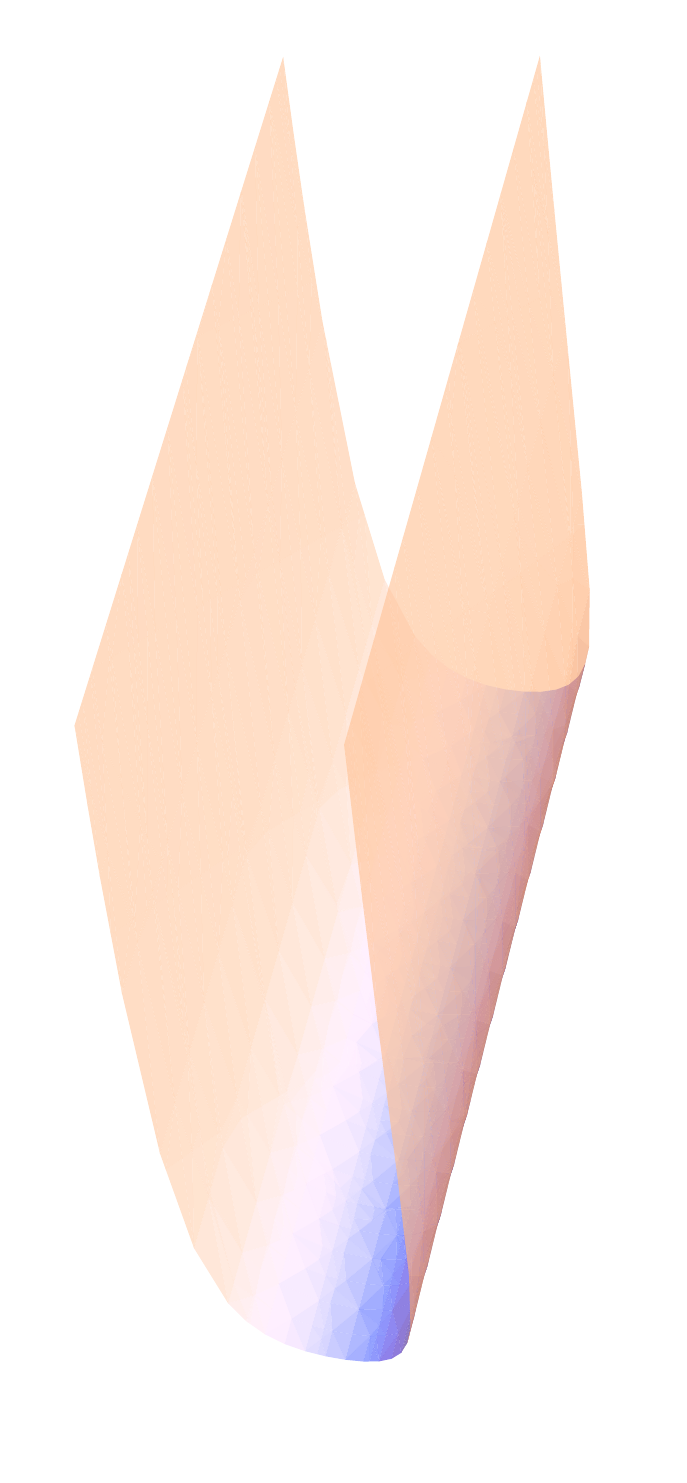} \ \ \ \
\includegraphics[width=0.4\linewidth]{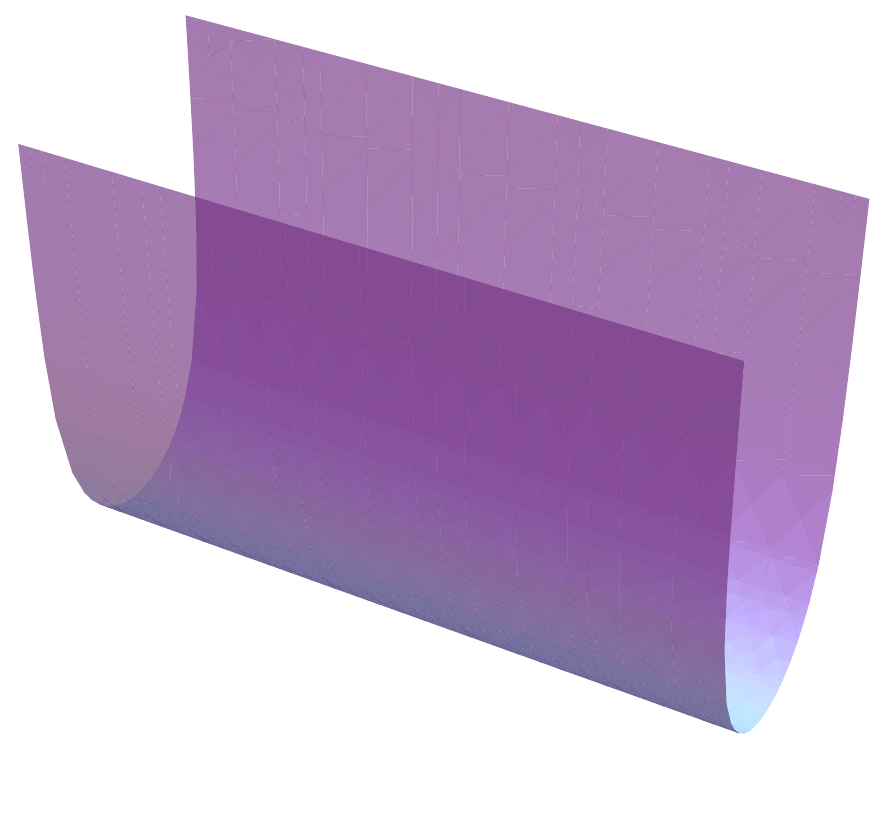} 
\end{center}
\caption{Tilted Grim-Reaper for $\lambda=.5$ and Grim-Reaper for $\lambda=0$.}\label{f11}
\end{figure}
%\begin{proof}
%From the Proposition (\ref{examplesL3}) and the Calabi's correspondence, the Gauss-Kronecker curvature of $\widehat{\Psi}$ vanish. 
%\\
%
%Consequently, it is the parametrization of a Tilted-Grim reaper due to the classification of complete translating solitons in \cite{HIMW} and \cite{MSHS1}.
%\\
%
%On the other hand, by straightforward calculus, we prove the formula (\ref{TGReaper}) from (\ref{param}),(\ref{primerterm}) and (\ref{segundoterm}),
%\begin{equation}
%\label{primerterm}
%\vec{e}_{3}\wedge_{\mathbb{L}_{3}}(dX\wedge_{\mathbb{L}^{3}}N)  = (-\text{cosh}(\theta)\, ds,-y_{0}\, d\theta,0),
%\end{equation}
%\begin{equation}
%\label{segundoterm}
%-\ll d\Psi,\vec{e}_{3}\gg=y_{0}\text{sinh}(\theta)\, d\theta.
%\end{equation}
%\end{proof}
\section{Concluding remarks}
\begin{itemize}
\item Formulas \eqref{globalc} and \eqref{lglobalc} are the source of most the results in this paper and dubtless they will have other applications. For example, by using the classification of flat translating solitons in $\R^3$ given in \cite{HIMW} and \eqref{globalc} one can  give a complete description  of flat singular $(-1)$-maximal spacelike graphs in $\L^3$. In this sense it is interesting to classify complete flat $[\varphi,\vec{e}_3]$-minimal graphs in $\R^3$ at least when $\varphi$ is an strictly increasing (decreasing) diffeomorphism and apply our correspondence to give the corresponding classification result in $\L^3$.
\item The existence of $\alpha$-maximal  bowls of elliptic and hyperbolic type, motivates the interest for obtaining Bernstein type results for this kind of surfaces.
\item More generally, it would be interesting to know if the examples of winglike surfaces described in Theorem \ref{texistencia2} are the only ones either translating solitons or singular $\alpha$-maximal surfaces which are graphs on the  punctured plane $\R^2\backslash \{(0,0)\}$.
\end{itemize}

\end{document}